\documentclass{article}

\usepackage[letterpaper,top=2cm,bottom=2cm,left=3cm,right=3cm,marginparwidth=1.75cm]{geometry}

\usepackage[T1]{fontenc}
\usepackage[utf8]{inputenc}
\usepackage[english]{babel}
\usepackage{amsmath,amsthm,amssymb,amsfonts,amscd, amsbsy,mathtools}
\usepackage{wasysym}
\usepackage[hidelinks]{hyperref}
\usepackage{pdfpages}

\usepackage{bussproofs}
\everymath{\displaystyle}
\DeclareMathAlphabet{\mathcal}{OMS}{cmsy}{m}{n}

\usepackage{faktor} %Por exemplo, faktor é muito útil para escrever quocientes.

\newtheorem*{theorem*}{Teorema}
\newtheorem{theorem}{Theorem}

\newtheorem{corollary}[theorem]{Corollary}

\newtheorem{definition}{Definition}
\newtheorem{example}[theorem]{Example}

\newtheorem{lemma}[theorem]{Lemma}

\newtheorem{proposition}{Proposition}
\newtheorem{remark}[theorem]{Remark}

\title{On Laguerre Isotropic Hypersurfaces}

\author{
  Fernanda Alves Caixeta\textsuperscript{1}
  \and % Se tiver mais de um autor
  Keti Tenenblat\textsuperscript{2}
}
\date{}

\begin{document}

\maketitle % O comando para criar o título

% --- COLOQUE AS AFILIAÇÕES AQUI, LOGO DEPOIS DO \maketitle ---
\footnotetext[1]{Instituto Acadêmico de Educação e Licenciaturas - Universidade Estadual de Goiás (UEG); Partially suppoted by CAPES/Brazil-Finance code 001; e-mail: fernanda.caixeta@ueg.br}
\footnotetext[2]{Departamento de Matemática, Universidade de Brasília (UnB); Partially suppoted by CNPq Grant 311916/2021-0; CAPES/Brazil-Finance code 001 and FAPDF 00193-00001678/2024-39; e-mail: k.tenenblat@mat.unb.br}
%\author{Fernanda Alves Caixeta\, \, \,   Keti Tenenblat}

%\begin{document}
%\maketitle

\begin{abstract}
\hspace{0.5cm} We study 
Laguerre isotropic hypersurfaces in the Euclidean space, which are 
hypersurfaces whose  Laguerre form  is zero and the eigenvalues of 
the Laguerre tensor are constant and equal to 
$\lambda\geq 0$.
We prove  a rigidity theorem for  the 
L-isotropic hypersurfaces parametrized by lines of curvature. 
Moreover, we study the hypersurfaces that are 
L-isotropic  and  L-isoparametric simultaneously and 
we show  that for such a hypersurface  $\lambda=0$.
We obtain necessary 
conditions for the existence of L-isotropic 
hypersurfaces with $\lambda > 0$ and we prove that a certain function, determined by 
the radii of curvature of the hypersurface, is bounded  above 
by ${1}/{2\lambda}$.


\end{abstract}

\section{Introduction}

\hspace{0.5cm}The theory of Laguerre geometry began in the 20th century with the study of Laguerre surfaces in $\mathbb{R}^3$, as presented in Blaschke’s book \cite{Blaschke}. Since then, several mathematicians have dedicated themselves to the study of Laguerre differential geometry, including Musso and Nicolodi in \cite{Musso}, \cite{Musso2}, and Palmer in \cite{Palmer}. The study of Laguerre geometry of surfaces in $\mathbb{R}^{3}$ continued  with the work of Li \cite{Li superficies}. Later, Song and Wang \cite{superficies minimas} studied minimal Laguerre surfaces in $\mathbb{R}^{3}$. The development of Laguerre geometry for surfaces naturally motivated a generalization to the case of hypersurfaces in $\mathbb{R}^{n+1}$. Subsequently, several papers were published addressing this more general case.

In Laguerre differential geometry, one studies the properties of Laguerre hypersurfaces that are invariant under the group of Laguerre transformations on the unit tangent bundle $U\mathbb{R}^{n+1}$. Li and Wang \cite{Li pequeno} studied this geometry for hypersurfaces in $\mathbb{R}^{n+1}$ using Cartan’s method of moving frames. Considering an immersion $x: M^n \rightarrow \mathbb{R}^{n+1}$ without umbilic points, with distinct nonzero principal curvatures, and $\xi: M^n \rightarrow \mathbb{S}^{n}$ a unit vector field normal to $x$, they defined the basic Laguerre invariants, namely: a Laguerre-invariant metric $g$, the Laguerre second fundamental form $\mathbb{B}$, the Laguerre form $\mathbb{C}$ and the Laguerre tensor $\mathbb{L}$.

Once these invariants were defined it was natural to undertake the task of classifying hypersurfaces, in the sense of Laguerre geometry, according to the properties satisfied by these invariants. The first results in this direction were obtained by Li, Li, and Wang \cite{Li paralela} where the authors introduced the concept of a hypersurface with parallel Laguerre second fundamental form $\mathbb{B}$, i.e., $\nabla \mathbb{B} = 0$, where $\nabla$ is the covariant derivative with respect to the Laguerre metric $g$. Moreover, they obtained a complete classification of such hypersurfaces, up to Laguerre transformations.
%Isto é, mostraram que a menos de transformação de Laguerre, as únicas 
%hipersuperfícies com segunda forma fundamental $\mathbb{B}$ paralela em 
%$\mathbb{R}^{n+1}$ são as cíclides de Dupin ou a imagem pela imersão de
%Laguerre $\tau$ da hipersuperfície $x:M^n\rightarrow\mathbb{R}^{n+1}_0$ que
%corresponde, a menos de uma translação em 
%$\mathbb{R}^{n+1}$, a hipersuperfície descrita em (\ref{exemplo}) quando $\gamma=0$.

Subsequently, Li, Li, and Wang \cite{isotropicas} began to study a class of hypersurfaces in $\mathbb{R}^{n+1}$ with vanishing Laguerre form $\mathbb{C}$ and with all eigenvalues of the Laguerre tensor $\mathbb{L}$ constant equal to $\lambda$. We refer to hypersurfaces with this property as L-isotropic hypersurfaces. Li, Li, and Wang proved that $\lambda \geq 0$ and presented a characterization result for L-isotropic hypersurfaces. Moreover, they obtained an explicit classification of hypersurfaces with $\mathbb{C} = 0$ and constant eigenvalues of the tensor $\mathbb{L}$ that are not all equal.

Li and Sun \cite{Li e Sun em R4} defined another class of hypersurfaces, the L-isoparametric hypersurfaces, that are those for which the eigenvalues of the Laguerre second fundamental form $\mathbb{B}$ are constant and the Laguerre form $\mathbb{C}$ vanishes. They presented a classification of these hypersurfaces in the case of three-dimensional manifolds in $\mathbb{R}^{4}$. They also showed that an L-isoparametric hypersurface is a Dupin hypersurface. Conversely, a Dupin hypersurface, under the additional condition that the Laguerre form $\mathbb{C}$ vanishes, is an L-isoparametric hypersurface.

%Uma vez tendo em mãos o teorema de classificação das hipersuperfícies com segunda forma fundamental paralela, proposto por  Li, Li e Wang \cite{Li paralela}, a técnica utilizada por Li e Sun na  prova do teorema de classificação foi mostrar que as hipersuperfícies L-isoparamétricas em 
%$\mathbb{R}^{4}$ possuem segunda forma fundamental 
%paralela. Eles também observaram que as superfícies 
%L-isoparamétricas também possuem 
%segunda forma fundamental paralela. 
%A partir disso, começou um vasto estudo a fim de 
%obter a classificação completa das hipersuperfícies
%L-isoparamétricas em $\mathbb{R}^{n+1}$. 

Song \cite{isop 2 distintas}, and Cezana and Tenenblat \cite{Miguel}, obtained the classification of L-isoparametric hypersurfaces in $\mathbb{R}^{n+1}$ with two distinct, nonzero principal curvatures.
Subsequently, Li and Shu \cite{3distintas e uma simples} obtained the classification of L-isoparametric hypersurfaces with three distinct principal curvatures, one of them being simple.
%sendo Laguerre equivalente a imagem pela imersão de Laguerre $\tau$ da hipersuperfície $x:M^n\rightarrow\mathbb{R}^{n+1}_0$ que
%corresponde, a menos de uma translação em 
%$\mathbb{R}^{n+1}$, a hipersuperfície descrita em (\ref{exemplo}).

Shu published two other papers obtaining the classification of L-isoparametric hypersurfaces in $\mathbb{R}^5$ \cite{isop em R5} and in $\mathbb{R}^6$ \cite{isop em R6}, and he also proved that such hypersurfaces have parallel second fundamental, when they are not L-isotropic. More recently, Shu \cite{isop em Rn} studied L-isoparametric hypersurfaces and Dupin hypersurfaces in $\mathbb{R}^{n+1}$, presenting several classification theorems. In one of his results, considering an L-isoparametric hypersurface in $\mathbb{R}^{n+1}$, he proved that the eigenvalues of the Laguerre tensor $\mathbb{L}$ are all constant, possibly all equal or not.
When the eigenvalues of the Laguerre tensor $\mathbb{L}$ are all constant but not all equal, the classification was obtained by Li, Li, and Wang in \cite{isotropicas}.
Note that for the complete classification of L-isoparametric hypersurfaces, it is sufficient to determine which are the L-isotropic hypersurfaces.

Moreover, in \cite{isop em Rn}, using the classification theorem for L-isoparametric hypersurfaces, Shu presented the classification of proper Dupin hypersurfaces in $\mathbb{R}^{n+1}$ with the number of distinct principal curvatures $g \geq 2$. This result, in a certain sense, generalizes the theorem obtained by Cezana and Tenenblat \cite{Miguel dupin}, where they proved that proper Dupin hypersurfaces $M^n \subset \mathbb{R}^{n+1}$ with $n$ distinct, nonvanishing principal curvatures and constant Laguerre curvature, admitting a coordinate system by lines of curvature, are determined by $n$ constants. Moreover, they proved that any Dupin hypersurface under these conditions in $\mathbb{R}^{n+1}$ is described by
\begin{equation}\label{familia da dupin introdução}
x(u_1,\ldots,u_n) = (u_1,\ldots,u_n, 0) -
\dfrac{\sum_{i=1}^na_iu_i^2 + \phi}{\sum_{i=1}^na_i^2u_i^2 + 1}(a_1u_1,\ldots,a_nu_n,-1),
\end{equation}
where $\phi \in \mathbb{R}$ and the $a_i$ are distinct, nonzero constants.

This family of Dupin hypersurfaces was first obtained by Corro, Ferreira, and Tenenblat \cite{artigo do exemplo} in 1999 via a Ribaucour transformation of a hyperplane. Although this example was derived in a different context, it was applied to the Laguerre geometry setting, when $\phi = 0$, where it first appeared in \cite{Li paralela} as an example of an L-isoparametric hypersurface.

Since then, the family given by \eqref{familia da dupin introdução} has appeared in all classification results of L-isoparametric hypersurfaces. Furthermore, this family also provides an example of an L-isotropic hypersurface with all eigenvalues $\lambda$ of the Laguerre tensor $\mathbb{L}$ constant and equal to zero.

In this work, motivated by the family \eqref{familia da dupin introdução} and Shu’s article \cite{isop em Rn}, we investigate hypersurfaces that simultaneously possess both properties of being L-isotropic and L-isoparametric. More precisely, we obtain the following proposition:
\begin{proposition}\label{propo introducao....}
Let $x:M^n \rightarrow \mathbb{R}^{n+1}$, with $n \geq 3$, be a hypersurface without umbilic points, having distinct, nonvanishing principal curvatures. If $x$ is both L-isotropic and L-isoparametric, then all eigenvalues of the Laguerre tensor $\mathbb{L}$ are equal and zero, i.e., $\lambda = 0$.
\end{proposition}

We posed another natural question whether the family of hypersurfaces described in \eqref{familia da dupin introdução} is the only L-isotropic one when $\lambda = 0$. We partially answer this question positively by presenting a rigidity result for L-isotropic hypersurfaces in $\mathbb{R}^{n+1}$ parametrized by lines of curvature, up to Laguerre transformation.

Given a hypersurface $x:M^n \rightarrow \mathbb{R}^{n+1}$ with unit normal vector $\xi:M^n \rightarrow \mathbb{S}^n$, the Laguerre position vector $Y:M^n \rightarrow \mathbb{C}^{n+3} \subset \mathbb{R}^{n+4}_2$ is defined by
$Y=\rho (x\cdot\xi,-x\cdot\xi,\xi, 1)$,
where
\begin{equation}\label{oioiiiii....}
\rho = \sqrt{\sum_{i} (r - r_i)^2}
\end{equation}
is a function defined in terms of the curvature radii $r_i = \frac{1}{k_i}$ and the mean curvature radius $r = \frac{\sum_i r_i}{n}$ of $x$. We then prove our main result:
\begin{theorem}\label{Teorema rigidez introduçao}
Let $x:M^n \rightarrow \mathbb{R}^{n+1}$ be a hypersurface with $n$ distinct, nonvanishing principal curvatures, let $Y:M^n \rightarrow \mathbb{C}^{n+3} \subset \mathbb{R}^{n+4}_2$ be its Laguerre position vector with Laguerre metric $g = \langle dY, dY \rangle$, and let $\lambda$ be the eigenvalues of the tensor $\mathbb{L}$. Suppose $M^n$ is connected and admits a parametrization by lines of curvature with respect to the metric $g$. If $x$ is an L-isotropic hypersurface, then $\lambda = 0$ and $x$ is L-isoparametric. Moreover, up to Laguerre transformation, this hypersurface is equivalent to the one described in \eqref{familia da dupin introdução} for $\phi = 0$.
\end{theorem}

The proof of Theorem \ref{Teorema rigidez introduçao} clarifies that choosing the constant $\phi$ is equivalent to choosing the position vector $Y$, which implies an equivalence in the sense of Laguerre. Consequently, from the perspective of Laguerre geometry, the family of hypersurfaces described by \eqref{familia da dupin introdução} is equivalent, regardless of whether $\phi = 0$ or $\phi \neq 0$.

The classification of L-isotropic hypersurfaces is of significant importance, as it serves as a foundation for classifying all L-isoparametric hypersurfaces, based on the result proven by Shu in \cite{isop em Rn}. The case of L-isotropic hypersurfaces with $\lambda > 0$ remains an open area of investigation, as no examples are currently known. In this work, we establish necessary conditions for the existence of such hypersurfaces. Theorem \ref{Teorema rigidez introduçao} demonstrates that an L-isotropic hypersurface with $\lambda > 0$ cannot admit a parameterization by lines of curvature. Furthermore, we verify that any such hypersurface must possess at least three distinct, non-vanishing principal curvatures.

For the class of L-isotropic hypersurfaces with $\lambda > 0$, we also prove the following proposition:
\begin{proposition}\label{proposição 1 introdução}
Let $x:M^n \rightarrow \mathbb{R}^{n+1}$ be a hypersurface without umbilic points and with nonvanishing principal curvatures. If $x$ is an L-isotropic hypersurface with $\lambda > 0$, then
\begin{equation}
    0 < \rho^2 < \dfrac{1}{2\lambda},
\end{equation}
where $\rho$ is defined by \eqref{oioiiiii....}.
\end{proposition}

\section{Preliminares}
In this section, we recall the Laguerre invariant tensors and the structure equations for hypersurfaces in 
$\mathbb{R}^n$. For more details, we refer to 
\cite{Li pequeno}. 

Let $\mathbb{R}^{n+4}_2$ be the space $\mathbb{R}^{n+4}$ with the inner product

\begin{equation}\label{produto interno com dois menos}
\langle X, Y \rangle = -X_1Y_1 + X_2Y_2 +\cdots + X_{n+3}Y_{n+3} - 
X_{n+4}Y_{n+4}.
\end{equation} 
%Seja $\mathbb{R}^{n+4}_{2}$ o espaço $\mathbb{R}^{n+4}$ munido com a 
%métrica 
%(\ref{produto interno com dois menos}). 
We define the \textit{light cone} $C^{n+3}$ in $\mathbb{R}^{n+4}_{2}$ by
\begin{equation}\label{cone de luz}
C^{n+3}=\{X\in\mathbb{R}^{n+4}\, \, ;\, \,  \langle X,X\rangle=0\}.
\end{equation}
Let $L\mathbb{G}$ be the subgroup of the orthogonal group 
$\mathbb{O}(n+2,2)$ in
$\mathbb{R}^{n+4}_2$ given by
\begin{equation}
L\mathbb{G} = \{T\in\mathbb{O}(n+2,2)\, \, \, ;\, \, \,  PT=P\},
\end{equation}
where $P=(1, -1, \Vec{0},0)$, $\Vec{0}\in\mathbb{R}^{n+1}$, is a lightlike vector in  $\mathbb{R}^{n+4}_2$ and
$\mathbb{O}(n+2,2)$ is the orthogonal group preserving the inner product
(\ref{produto interno com dois menos}).

Let $x:M^n\longrightarrow\mathbb{R}^{n+1}$ be an orientable hypersurface without umbilic points
and with nonvanishing principal curvatures. 
Consider 
$\{e_1,e_2,\ldots, e_n\}$ an orthonormal basis with respect to the metric $dx\cdot dx$ for $TM$, consisting
of unit principal vectors, i.e.,
\begin{equation}\label{equacao de estrutura da x hiper. orientada}
e_i(\xi) = -k_ie_{i}(x),
\, \, \, 1\leq i\leq n,
\end{equation}
where $k_i$ are the nonzero principal curvatures corresponding to $e_i$. Define
\begin{equation}\label{estou com fomeeee}
r_i=\dfrac{1}{k_i},\, \, \, \, \,  \, \, \, \, \, \,   r=\dfrac{r_1+r_2+\cdots +r_n}{n}, \, \, \, \,  \, \, \, \, \, \,  \rho = \sqrt{\sum_{i}(r-r_i)^2},
\end{equation}
where $r_i$ is the \textit{radius of curvature} and $r$ the \textit{mean radius of curvature} of $x$.

We define the map
\begin{equation}\label{definicao da Y}
Y=\rho (x\cdot\xi,-x\cdot\xi,\xi, 1): M^n\longrightarrow C^{n+3}\subset\mathbb{R}^{n+4}_2,
\end{equation} 
where $C^{n+3}$ is given by \eqref{cone de luz} and 
$\rho$ by \eqref{estou com fomeeee} and $x\cdot\xi$ denotes the  euclidean inner product. The vector
$Y$ is called the \textit{Laguerre position vector of the immersion $x$}, which provides the imersion of the hypersurface  $x$ into the light cone of $\mathbb{R}^{n+4}_2$.
Note that $\rho$ never vanishes, since $x$ has no umbilic points. Moreover, we can write
\begin{equation}\label{relacao entre Y e y}
Y=\rho y,\, \, \, \, \, \, \text{where}\, \, \, \, \, \,  y = (x\cdot\xi,-x\cdot\xi,\xi, 1).
\end{equation}

\begin{theorem}(\cite{Li pequeno})
\label{teorema que caracteriza equivalente por t.laguerre com y}
Let $x,x^*: M^n\longrightarrow\mathbb{R}^{n+1}$ be two oriented hypersurfaces
with nonvanishing principal curvatures. Then $x$ and $x^*$ are
equivalent under Laguerre transformations if, and only if, there exists $T\in L\mathbb{G}$ such that
$Y^* = YT$, where $Y$ and $Y^*$ are the Laguerre positions vectors of $x$ and $x^*$ respectively.
\end{theorem}
Consider $y=Y/\rho$ given by \eqref{relacao entre Y e y}. Let
$III=\langle dy,dy \rangle$ be the
third fundamental form of  $x$. Let
$\{E'_1, \ldots, E'_n \} $ be an orthonormal basis for
$III = \langle dy,dy\rangle= d\xi\cdot d\xi$, where the elements
of this basis can be written as $E'_{i} = r_ie_i $,  $1\leq i\leq n $.

The \textit{Laguerre metric of  $x$} is defined by,
\begin{equation}\label{metrica de Laguerre}
g = \left(\sum_{i}(r-r_i)^2 \right)III
= \rho^2 III,
\end{equation}
where $r_i$ and $r$ are, respectively, the radius of curvature and the mean radius of curvature of  $x$.
It follows from Theorem \ref{teorema que caracteriza equivalente por t.laguerre com y} 
that the metric \eqref{metrica de Laguerre} is a Laguerre invariant.

Let $\Delta_{g}$  be the Laplacian operator of the metric $g$, 
we define
\begin{equation}\label{definicao do N}
N = \dfrac{1}{n}\Delta_{g} Y + \dfrac{1}{2n^2}\langle \Delta_{g} Y, \Delta_{g} Y \rangle Y,
\end{equation}
and
\begin{eqnarray}\label{aplicacao normal de laguerre}
\eta 
&=&\left(\dfrac{1}{2}(1+ \vert x \vert ^2), 
\dfrac{1}{2}(1- \vert x\vert ^2), x, 0 \right) + r(x\cdot\xi, - x\cdot\xi, \xi, 1).
\end{eqnarray}
The map $\eta: M^n\longrightarrow C^{n+3}\subset\mathbb{R}^{n+4}_2$ 
defined by
(\ref{aplicacao normal de laguerre}) is called the
\textit{Laguerre normal map}, where
$C^{n+3}$  is the light cone given in
\eqref{cone de luz}. 
The following relations are satisfied in the metric
(\ref{produto interno com dois menos}),
\begin{equation}\label{relacoes de ortogonalidade para Y}
\langle Y , Y \rangle = \langle N , N \rangle =0,\, \, \, 
\langle Y, N\rangle = -1,\, \, \, 
 \langle \eta, \eta\rangle = 0,
\langle \eta , P \rangle = -1.
\end{equation}
Let $\{E_1, \ldots, E_n \} $ be vector fields in $\mathbb{R}^{n+4}_2$ tangent to $Y$ consisting
of an orthonormal basis with respect to the metric $g$
with dual basis $\{\omega_1, \ldots, \omega_n\} $.
Then we have the following orthogonal decomposition
\begin{equation*}
\mathbb{R}^{n+4}_{2} = span\{Y,N \}\oplus span\{ E_{1}(Y), \ldots, E_{n}(Y)\}\oplus
\{\eta, P \}.
\end{equation*}

We call $\{Y,N, E_{1}(Y), \ldots, E_{n}(Y), \eta, P  \} $ 
the
\textit{Laguerre moving frame} in $\mathbb{R}^{n+4}_{2}$ of $x$.
The structure equations for the frame $\{Y,N, E_{1}(Y), \ldots, E_{n}(Y), \eta, P \} $ are written as
\begin{equation}\label{equacoes de estrutura menos gerais}
\begin{array}{lc}
E_{i}(N) = \displaystyle\sum _{j} L_{ij} E_{j}(Y) + C_i P,\\
E_{j}(E_{i}(Y)) = L_{ij}Y + \delta_{ij}N + \displaystyle\sum _{k}\Gamma^{k}_{ij} E_{k}(Y) + B_{ij}P,\\
E_{i}(\eta) =  -C_iY + \displaystyle\sum _{j} B_{ij} E_{j}(Y),
\end{array}
\end{equation}
where $\Gamma ^{k}_{ij} = \omega_{ik}(E_j)$ are 
differentiable functions defined on $M^n$.
From equations  \eqref{equacoes de estrutura menos gerais} we define the following Laguerre invariants.
\begin{definition}\label{definicao dos tensores B L e C}
\rm The Laguerre invariant tensors
$\mathbb{B}=\sum_{ij}B_{ij}\omega_{i}\otimes\omega_{j}$ and 
$\mathbb{L}=\sum_{ij}L_{ij}\omega_{i}\otimes\omega_{j}$ are called the \textit{Laguerre second fundamental form} and the  
\textit{Laguerre tensor},
respectively. We also define the form
$\mathbb{C}=\sum_{i}C_{i}\omega_{i}$, which is called \textit{Laguerre form}. 
\end{definition}
Taking derivatives of the structure equations
\eqref{equacoes de estrutura menos gerais},
we obtain the following relations among the invariants:

\begin{equation}\label{relacao entre os invariantes L, B e C}
\begin{array}{lc}
L_{ij,k} = L_{ik,j},\\
C_{i,j} - C_{j,i} = \displaystyle\sum_{k}\left( B_{ki}L_{kj} - B_{jk}L_{ik}\right),\\
B_{ij,k} - B_{ik,j} = C_{j}\delta_{ik} - C_{k}\delta_{ij},\\
R_{ijkl} = L_{jk}\delta_{il} + L_{il}\delta_{jk} - L_{ik}\delta_{jl} - L_{jl}\delta_{ik},
\end{array}
\end{equation}
where $\{L_{ij,k}\}$, $\{C_{i,j}\}$, 
$\{B_{ij,k}\}$  are the covariant derivatives
of the components of the tensors  $\mathbb{L}$, 
$\mathbb{B}$ and $\mathbb{C}$  with respect to the
metric $g$ and
$R_{ijkl}$ is the curvature tensor of the metric $g$.
Furthermore, the following identities are satisfied:
\begin{equation}\label{identidades que os B satisfazem}
\sum_{ij}B^2_{ij} = 1,\, \, \,  \, \, \,  \, \, \, \sum_{i}B_{ii} = 0,
\, \, \, \, \, \,  \, \, \,   \sum_{i}L_{ii}= -\dfrac{1}{2n}\langle 
\Delta_{g}Y, \Delta_{g}Y \rangle,\, \, \,  \, \,\, \, \sum_{i}B_{ij,i} = (n-1)C_{j}.
\end{equation}

Consider $E'_{i} =r_{i}e_{i}$, $1\leq i\leq n$ an orthonormal basis for $d\xi\cdot d\xi$. Then,
the fields $E_{i}=\rho^{-1}E'_{i}$, $1\leq i\leq n$, 
form an orthonormal basis in the Laguerre metric
$g=\rho^{2}d\xi\cdot d\xi$ with dual basis $\omega_{i}$. Thus, the following equalities hold
\begin{eqnarray}\label{expressao dos B e C em termos dos ei}
\begin{array}{lc}
B_{ij}=\rho^{-1}(r-r_{i})\delta_{ij},\\
\\
C_{i} = -\rho^{-2}r_{i}\{e_{i}(r) - (r-r_{i})e_{i}(log\rho) \},
\end{array}
\end{eqnarray}
\begin{equation}\label{expressao dos LIJ na observacao}
    L_{ij} = \rho^{-2}\bigg\{(Hess_{ij}(log\rho))_{III} -
r_ir_je_{i}(log\rho)e_{j}(log\rho) + \dfrac{1}{2}\bigg(   
\vert \nabla_{III}log\rho\vert^2 - 1 \bigg)\delta_{ij}\bigg\},
\end{equation}
where $Hess_{ij}$ and $\nabla_{III}$ are respectively the element $ij$ of the Hessian matrix and the gradient with respect to the third fundamental form 
$III=d\xi.d\xi$ of $x$. 
\begin{remark}
  \rm  From now on, we use the notation
    $E_i$, $E'_i$ and $e_i$ to refer to the orthonormal basis in the metrics  $g$, 
    $III=  d\xi\cdot d\xi$ and $ dx\cdot dx$, respectively. 
The relation among these basis is given by the expression \begin{equation}\label{relacao entre as bases de g e III}
E_i = \rho^{-1}E'_{i} = \rho^{-1}r_ie_i , \, \, \,  1\leq i\leq n.
\end{equation}
\end{remark}
We denote by $b_i$ the eigenvalues of the Laguerre second fundamental form $\mathbb{B}$, which are called \textit{Laguerre principal curvatures} of $x$. Then, from (\ref{expressao dos B e C em termos dos ei}),
the Laguerre principal curvatures are
\begin{equation}\label{curvaturas principais de Laguerre}
b_{i} = \rho^{-1}(r - r_{i}).
\end{equation}
The following proposition was obtained by Song, Y.P. in
\cite{isop 2 distintas}.
\begin{proposition}\label{propo. do Song duas ki}
(\cite{isop 2 distintas})
Let $x:M^n\rightarrow\mathbb{R}^{n+1}$ be a hypersurface without umbilic points, with two distinct nonvanishing principal 
curvatures $k_1$ e $k_2$ ($k_2<k_1$) with multiplicities $m$ e $n-m$, respectively. Then the Laguerre principal curvatures are constant given by
\begin{equation}
b_{1} =  \sqrt{\dfrac{(n-m)}{mn}}, \,  \,  \,  \,  \,  \, b_{2} = - \sqrt{\dfrac{m}{n(n-m)}}.
\end{equation}
\end{proposition}

\section{L-Isotropic Hypersurface in \texorpdfstring{$\mathbb{R}^n$}{Rn}}

In this section, we define  
L-isotropic hypersurfaces  
in $\mathbb{R}^{n+1}$ and we show  
that the function $\rho^2$, defined in terms of the curvature radio of the immersion $x$ by \eqref{oioiiiii....}, is bounded above by $\dfrac{1}{2\lambda}$, where $\lambda$ is the common value of the eigenvalues of the Laguerre tensor $\mathbb{L}$.

\begin{definition}
\rm
Let $x:M^n\rightarrow\mathbb{R}^{n+1}$ be a  
hypersurface without umbilic points, with nonvanishing  
principal curvatures. We say that $x$ is an  
\textit{L-isotropic hypersurface}  
if the Laguerre form $\mathbb{C}$ vanishes and  
the eigenvalues of the Laguerre tensor  
$\mathbb{L}$ are all equal to a constant, i.e.,  
$L_{ij}= \lambda\delta_{ij}.$
\end{definition}
Let $x$ be L-isotropic. It follows from \eqref{relacao entre os invariantes L, B e C}  
that the tensors $\mathbb{B}$ and $\mathbb{L}$ commute.  
Thus, we can choose an orthonormal frame $\{E_1, \ldots, E_n\}$ such that,

\begin{equation}\label{tensores B e L comutando ingles..}
B_{ij} = b_i\delta_{ij},\, \, \, \, \,  L_{ij} = \lambda\delta_{ij},
\end{equation}
where $b_i$, $i=1,\ldots,n$ are the Laguerre principal curvatures.
\begin{proposition}\label{lambra maior ou igual a zero}(\cite{isotropicas})
Let $x:M^n\rightarrow\mathbb{R}^{n+1}$ be an  
oriented hypersurface without umbilic points, with nonvanishing  
principal curvatures. If $x$ is  
an L-isotropic hypersurface, then  
the eigenvalues of the tensor $\mathbb{L}$ are  
non-negative, i.e., $\lambda\geq 0$. Moreover,  
the Laguerre second fundamental form $\mathbb{B}$ and the components of the curvature tensor $R_{ijkl}$  
with respect to the Laguerre metric $g$ satisfy  
the following relations,
\begin{eqnarray} \label{bij,k = bik,j coreccao...}
&&B_{ij,k} = B_{ik,j},\, \, \, \, \, i,j,k = 1,\ldots,n, \\ \label{Rijkl da isotropica correcaoo.}
&& R_{ijkl} = 2\lambda (\delta_{jk}\delta_{il} - \delta_{ik}\delta_{jl}), \\ \label{bij,k = 2nlambda...}
&&\sum_{ijk}(B_{ij,k})^2 = 2n\lambda.
\end{eqnarray}

\end{proposition}

The following corollary follows from the proof of Proposition \ref{lambra maior ou igual a zero}:

\begin{corollary} \label{corolario labda=0 segunda paralela}
If $x$ is an L-isotropic hypersurface, then $\nabla\mathbb{B} = 0$ if and only if $\lambda = 0$, i.e., the eigenvalues of $\mathbb{B}$ are constant if and only if $\lambda = 0$.
\end{corollary}

As a consequence of Corollary \ref{corolario labda=0 segunda paralela} and Proposition \ref{propo. do Song duas ki}, we have:

\begin{proposition}\label{observacao l-isotrpoicas lambda=0}
Let $x:M^n\rightarrow\mathbb{R}^{n+1}$ be a hypersurface without umbilic points, with two distinct nonvanishing principal curvatures. If $x$ is L-isotropic, then $\lambda = 0$.
\end{proposition}

\begin{proof}
Indeed, by Proposition \ref{propo. do Song duas ki}, we have that $b_1$ and $b_2$ are constant with multiplicities $m$ and $(n - m)$, respectively.  
Thus, we have $\nabla\mathbb{B} = 0$. Since $x$ is an L-isotropic hypersurface, it follows from Corollary \ref{corolario labda=0 segunda paralela} that $\lambda = 0$.
\end{proof}

\begin{corollary}\label{mais uma observacao sobre lambra positivo}
If $x:M^n\rightarrow\mathbb{R}^{n+1}$ ($n\geq 3$) is an L-isotropic hypersurface with $\lambda > 0$, then $x$ has at least three distinct 
nonvanishing principal curvatures.
\end{corollary}

In what follows, we will prove that
the function $\rho^2$ is bounded above by $\frac{1}{2\lambda}$, where $\lambda>0$ are the eigenvalues of the Laguerre tensor $\mathbb{L}$.
More precisely,

\begin{proposition}\label{proposição limitacao do lambda}
Let $x:M^n\rightarrow\mathbb{R}^{n+1}$ be a 
hypersurface without umbilic points, with nonvanishing principal curvatures. If $x$ is
L-isotropic with $\lambda > 0$, then
\begin{equation}\label{expressao de limitacao do lambra}
    0 < \rho^2 < \dfrac{1}{2\lambda},
\end{equation}
where $\rho$ is the positive function defined in terms of the 
principal curvatures and the curvature radii of $x$, given by
\eqref{estou com fomeeee}.
\end{proposition}
\begin{proof}
Let $x$ be an L-isotropic hypersurface.
From the first equation of \eqref{equacoes de estrutura menos gerais}, we have that,

\begin{equation}\label{N = lambda Y}
    N = \lambda Y + \alpha,
\end{equation}
where $\alpha\in\mathbb{R}^{n+4}_2$ is a constant.
From the orthogonality relations given by
\eqref{relacoes de ortogonalidade para Y},
we obtain properties of the constant 
$\alpha$. Since $\langle N,Y \rangle =-1$, $\langle Y, Y\rangle = 0$ and $\langle N, P\rangle = 0$, we have from \eqref{N = lambda Y} that,
\begin{equation}
 \langle \alpha, Y \rangle = -1,\, \, \, \langle \alpha, \alpha\rangle= 2\lambda,\, \, \, \langle \alpha,P\rangle =0.    
\end{equation} 
Up to a Laguerre transformation, we may consider $\alpha
=(0,0,\sqrt{2\lambda},\vec{0},0)$, $\vec{0}\in\mathbb{R}^n$.

Using the relation \eqref{identidades que os B satisfazem} of the tensor $\mathbb{L}$, we have,
\begin{eqnarray}\label{traço da L com a hip.}
 &&\lambda = -\dfrac{1}{2n^2}\langle 
\Delta_{g}Y, \Delta_{g}Y \rangle.
\end{eqnarray}
From the properties \eqref{definicao do N}, \eqref{N = lambda Y} and \eqref{traço da L com a hip.} we obtain
%\begin{eqnarray}\label{laplaciano de Y = 2lambda}\nonumber
%&&  \lambda Y + \alpha = \dfrac{1}{n}\Delta_{g} Y + \dfrac{1}{2n^2}\langle \Delta_{g} %Y, \Delta_{g} Y \rangle Y  = \dfrac{1}{n}\Delta_{g} Y - \lambda Y.
%\end{eqnarray}
%Therefore, 
\begin{eqnarray}\label{laplaciano do Y na correcao....}
     \Delta_{g} Y = 2n\lambda Y + n\alpha.
\end{eqnarray}
The relation \eqref{laplaciano do Y na correcao....} was originally proved by Li, T.; Li, H. and Wang, C. in \cite{isotropicas}.

In what follows,  
we will obtain $\Delta_{g}Y$.  
Considering a geodesic frame at a point $p$, by \eqref{relacao entre as bases de g e III} and \eqref{relacao entre Y e y} we have,  

\begin{eqnarray}\label{expressao do laplaciano de Y na terceira}\nonumber
&&\Delta_{g}Y = \sum_{t}E_{t}\Big(E_{t}(Y)\Big) = \sum_{t}\Big[\rho^{-1}r_te_t\Big(\rho^{-1}r_te_t(\rho y)\Big)\Big]\\\nonumber
%&&=  \sum_{t}\Bigg[\rho^{-1}r_te_t
%\Big(\rho^{-1}r_te_t(\rho)y+\rho^{-1}\rho r_te_t(y)\Big)\Bigg]\\
%&&= \sum_{t}\Bigg[\rho^{-1}r_te_t
%\Big(r_te_t(log\rho)y+ r_te_t(y)\Big)\Bigg]\\
%&&= \sum_{t}\Bigg[\rho^{-1}r_t \Bigg( e_t(r_t) e_t(log\rho)y + %r_te_t(e_t(log\rho))y + r_t e_t(log\rho)e_t(y) +e_t(r_t)e_t(y) + \\
%&&+ r_te_t(e_t(y))\Bigg)\Bigg]\\
&&=\rho^{-1}\Bigg\{ \sum_{t}\Bigg[ \Big( r_te_t(r_t) e_t(log\rho) + r_t^2e_t(e_t(log\rho))\Big)y + 
\Big(r_te_t(r_t)e_t(y)+ r_t^2e_t(e_t(y))\Big)+\\ \nonumber
&&+ r_t e_t(log\rho)r_te_t(y)\Bigg]\Bigg\} \\ &&=\rho^{-1}y\Delta_{III}(log\rho) + 
\rho^{-1}\Delta_{III}y + 
\rho^{-1}\nabla_{III}(log\rho)\nabla_{III}y.
\end{eqnarray}
%Hence,
%\begin{equation}
%\Delta_{g}Y = \rho^{-1}\Delta_{III}(\log\rho)y + 
%\rho^{-1}\Delta_{III}y + 
%\rho^{-1}\nabla_{III}(\log\rho) \cdot \nabla_{III}y.
%\end{equation}
Substituting \eqref{laplaciano do Y na correcao....} into 
\eqref{expressao do laplaciano de Y na terceira} and
using \eqref{relacao entre Y e y} we obtain,

\begin{eqnarray}\label{expressao y na terceira}
&& 2n\lambda\rho^2 y + n\rho\alpha = \Delta_{III}(log\rho)y + \Delta_{III}y + 
\nabla_{III}(log\rho)\nabla_{III}y.
\end{eqnarray}

Since $y, \alpha \in \mathbb{R}^{n+4}_2$, we can consider each coordinate of the expression 
\eqref{expressao y na terceira}, i.e.,
\begin{equation}\label{expressao em coordenadas da III}
    2n\lambda \rho^2 y^{l}+ n\rho\alpha^{l} = \Delta_{III}(log\rho)y^{l} + \Delta_{III}y^{l} + 
\nabla_{III}(log\rho)\nabla_{III}y^{l},\, \, \, \, l=1,\ldots, (n+4).
\end{equation}
By considering the coordinate $l = (n+4)$, it follows from the definition of $y$ that $y^{n+4} = 1$.  
Moreover, by the choice of $\alpha$, $\alpha^{l+4} = 0$.  
Then, it follows from \eqref{expressao em coordenadas da III} that

\begin{equation}\label{expressao em coordenadas da III para l=n+4}
     \Delta_{III}(log\rho) = 2n\lambda \rho^2.
\end{equation}

Now, considering the trace of  
\eqref{expressao dos LIJ na observacao}, it follows that,

\begin{eqnarray}\label{traço da expressao Lij com isotropica}
%&&\Delta_{III}(log\rho) - \sum_{i}
%r_i^2e_{i}(log\rho)^2 + \dfrac{n}{2}\bigg(   
%\vert \nabla_{III}log\rho\vert^2 - 1 \bigg) = n\lambda\rho^{2}\\
&&\Delta_{III}(log\rho) + \dfrac{(n-2)}{2}  
\vert \nabla_{III}log\rho\vert^2 - \dfrac{n}{2}   = n\lambda\rho^{2}.
\end{eqnarray}
Substituting \eqref{expressao em coordenadas da III para l=n+4} into \eqref{traço da expressao Lij com isotropica}, we have

\begin{eqnarray}\label{gradiente de pho}
%&&2n\lambda \rho^2 + 
%\dfrac{(n-2)}{2}  
%\vert \nabla_{III}log\rho\vert^2 - 
%\dfrac{n}{2}   = n\lambda\rho^{2}\\ 
&&\dfrac{(n-2)}{2} \vert \nabla_{III}log\rho\vert^2
= -n\lambda\rho^{2} +\dfrac{n}{2}.
\end{eqnarray}
Therefore, using Corollary \ref{mais uma observacao sobre lambra positivo}, it follows from \eqref{gradiente de pho} that equation 
\eqref{expressao de limitacao do lambra} is satisfied.

\end{proof}
The expression \eqref{bij,k = 2nlambda...} of Proposition \ref{lambra maior ou igual a zero} shows that L-isotropic hypersurfaces may exist in $\mathbb{R}^{n+1}$ only if $\lambda \geq 0$. A natural question that arises is: are there hypersurfaces belonging to this class? For the case $\lambda = 0$, one example is known. It was first obtained in 1999 by Corro, Ferreira, and Tenenblat \cite{artigo do exemplo}, where the authors applied Ribaucour transformations to the hyperplane to obtain a family of Dupin hypersurfaces in $\mathbb{R}^{n+1}$ parameterized by lines of curvature with n distinct principal curvatures. Although obtained in a different context, this example has been considered within the framework of Laguerre geometry.

With the objective of presenting the example mentioned above, we will introduce some definitions. For more details, see \cite{Li pequeno}.
Let $\mathbb{R}^{n+2}_{1}$ be the Minkowski space with the inner product
\begin{eqnarray}\label{imersao de laguerre produto intero}\nonumber
\langle X, Y\rangle = X_1Y_1+ \cdots +  X_{n+1}Y_{n+1} -  X_{n+2}Y_{n+2}.
\end{eqnarray}
Let $\nu= (1,\vec{0},1)$ be the lightlike vector in $\mathbb{R}^{n+2}_{1}$ with $\vec{0}\in\mathbb{R}^{n}$.
Consider $\mathbb{R}^{n+1}_{0}$ the degenerate hyperplane in $\mathbb{R}^{n+2}_{1}$ defined by 
$\mathbb{R}^{n+1}_{0} = \{x\in\mathbb{R}^{n+2}_{1}\, ;\,  \langle x,\nu\rangle =0\}$.
%\begin{eqnarray*}
%\mathbb{R}^{n+1}_{0} = \{x\in\mathbb{R}^{n+2}_{1}\, ;\,  \langle x,\nu\rangle =0\}.
%\end{eqnarray*} 
We define the bundle $U\mathbb{R}^{n+1}_0$ by
\begin{eqnarray}\label{imersao de Laguerre relacao do UR0}\nonumber
U\mathbb{R}^{n+1}_{0} = \{(x,\xi)\in\mathbb{R}^{n+2}_{1}\times\mathbb{R}^{n+2}_{1}\,  ;\,  
\langle x,\nu\rangle =0, \langle \xi,\xi\rangle =0, \langle \xi,\nu\rangle = 1 \}.
\end{eqnarray}
Note that, $\langle \xi,\nu\rangle = 1$
implies $\xi_1\neq 0$.

We define the \textit{Laguerre immersion} $\tau: U\mathbb{R}^{n+1}_{0}\rightarrow U\mathbb{R}^{n+1}$ 
given by $\tau(x,\xi)= (x',\xi')\in U\mathbb{R}^{n+1}$,
%\begin{eqnarray*}
%\tau(x,\xi)= (x',\xi')\in U\mathbb{R}^{n+1},
%\end{eqnarray*}
where $x = 
(x_1,\vec{x_0},x_1)\in\mathbb{R}\times\mathbb{R}^{n}\times\mathbb{R}$,
$\xi=(\xi_1 + 1,\vec{\xi_0},\xi_1)\in\mathbb{R}\times\mathbb{R}^{n}\times\mathbb{R}$
and
\begin{eqnarray}\label{imersao de Laguerre: relacao x e x'}
x' = \left(-\dfrac{x_1}{\xi_1}, \vec{x_0} - \dfrac{x_1}{\xi_1}\vec{\xi_0}\right),\, \, \, \,  
\xi' = \left(1 + \dfrac{1}{\xi_1}, \dfrac{\vec{\xi_0}}{\xi_1}\right).
\end{eqnarray}
Note that $\xi'$ is normal to $x'$. 
Let $x: M^n\rightarrow\mathbb{R}^{n+1}_{0}$ be a spacelike oriented hypersurface in the degenerate space $\mathbb{R}^{n+1}_{0}$ and let
$\xi$ be the unique vector in $\mathbb{R}^{n+2}_1$ satisfying
$\langle \xi, dx \rangle = 0,\, \, \langle \xi, \xi \rangle = 0\, \, \, \,  
\text{and}\, \, \, \,  \langle \xi, \nu \rangle = 1,$
%\begin{eqnarray*}
%\langle \xi, dx \rangle = 0,\, \, \langle \xi, \xi \rangle = 0\, \, \, \,  
%\text{and}\, \, \, \,  \langle \xi, \nu \rangle = 1,
%\end{eqnarray*}
where $\nu=(1,\vec{0},1)\in\mathbb{R}^{n+2}_1$ is the lightlike vector with $\vec{0}\in\mathbb{R}^{n}$. Consider $\{e_1,\ldots,e_n \}$ an orthonormal basis for $TM$ with respect to 
$\langle dx,dx\rangle$ such that $e_i(\xi) = -k_ie_i(x),\, \, 1\leq i\leq n.$
%\begin{equation}\label{equacao de imersoes 1}
%e_i(\xi) = -k_ie_i(x),\, \, 1\leq i\leq n. 
%\end{equation}
Assuming that $k_i\neq 0$,
we define $r_i=\dfrac{1}{k_i}$ as a curvature radius of $x$ and
$r=\dfrac{\sum_{i=1}^n r_i}{n}$ as the mean curvature radius of $x$.
We can relate the immersion $x$ with the corresponding hypersurface
$x': M^n\rightarrow\mathbb{R}^{n+1}$
as follows. Given $x: M^n\rightarrow\mathbb{R}^{n+1}_{0}$,
using the Laguerre immersion $\tau$
we obtain a hypersurface
$x': M^n\rightarrow\mathbb{R}^{n+1}$, where $x'$ and $\xi'$ are given by
(\ref{imersao de Laguerre: relacao x e x'}). 
Next, we present the example of an L-isotropic hypersurface.

\vspace{0.2in}

\noindent \begin{example}(\cite{artigo do exemplo})\label{exemplo da isotropica lambda =0}
\rm For any integers $m_1,m_2,\ldots, m_s$, with
$m_1+m_2+\cdots + m_s = n$ ($s\geq 2$) and distinct nonzero constants $a_1, a_2, \ldots, a_s$,
we define
$x:\mathbb{R}^{n}\rightarrow\mathbb{R}^{n+1}{0}$ as a spacelike oriented hypersurface in
$\mathbb{R}^{n+1}_{0}$ given by 
%$x = \left(\dfrac{1}{2}\sum_{i=1}^sa_i\vert u_i\vert^2, u_1, \ldots, u_s, 
%\dfrac{1}{2}\sum_{i=1}^sa_i\vert u_i\vert^2 \right),$
\begin{eqnarray*} 
x = \left(\dfrac{1}{2}\displaystyle\sum_{i=1}^sa_i\vert u_i\vert^2, u_1, \ldots, u_s, 
\dfrac{1}{2}\displaystyle\sum_{i=1}^sa_i\vert u_i\vert^2 \right),
\end{eqnarray*}
where $(u_1, u_2,\ldots, u_s)\in\mathbb{R}^{m_1}\times\mathbb{R}^{m_2}\times\cdots\times\mathbb{R}^{m_s}
=\mathbb{R}^{n}$, $\vert u_i\vert^2 = u_i\cdot u_i$, $i=1, \ldots, s$ and
\begin{eqnarray}\label{componentes da x}
\begin{array}{lc}
x_0 = (u_1,\ldots, u_s),\, \, \, \, \, \, \, \,  x_1 = 
\dfrac{1}{2}\sum_{i=1}^sa_i\vert u_i\vert^2.
\end{array}
\end{eqnarray}
The unit normal vector field $\xi$ of $x$ is given by, 
%$\xi = -\left(\dfrac{1}{2}\sum_{i=1}^s\vert a_i u_i\vert^2 + 1, a_1 u_1,\ldots,a_s u_s, 
%\dfrac{1}{2}\sum_{i=1}^s\vert a_i u_i\vert^2 -1\right),$
\begin{eqnarray*} 
\xi = -\left(\dfrac{1}{2}\displaystyle\sum_{i=1}^s\vert a_i u_i\vert^2 + 1, a_1 u_1,\ldots,a_s u_s, 
\dfrac{1}{2}\displaystyle\sum_{i=1}^s\vert a_i u_i\vert^2 -1 \right),
\end{eqnarray*}
where
\begin{eqnarray}\label{componentes da xi}
\begin{array}{lc}
\xi_0 = (- a_1 u_1,\ldots,-a_s u_s),\, \, \, \, \, \, \, \, \xi_1 = 
\dfrac{-1}{2}\sum_{i=1}^s\vert a_i u_i\vert^2 -1.
\end{array}
\end{eqnarray}
We have,
\begin{eqnarray*}
dx = \left(\sum_{i=1}^s a_i u_idu_i
\, \, ,\, \, 
du_1, \cdots, du_s\, \, ,\, \,  
\sum_{i=1}^s a_i u_idu_i\right),\, \, \, \, d\xi = -
\left(\sum_{i=1}^s a_i^2 u_idu_i, a_1 du_1, \cdots, a_s du_s, 
\sum_{i=1}^s a_i^2 u_idu_i \right).
\end{eqnarray*} 
The first and second fundamental forms of $x$
are, respectively, given by
\begin{eqnarray*}
I = dx\cdot dx = \displaystyle\sum_{i=1}^s du_idu_i,\, \, \, \, \, 
II = -dx\cdot d\xi =  \displaystyle\sum_{i=1}^s a_i du_idu_i.
\end{eqnarray*}
Note that $x$ is parametrized by curvature lines and its principal curvatures are
$a_1, \ldots, a_s$ with multiplicities $m_1, m_2, \ldots,m_s$. 
Hence, $r_i = \dfrac{1}{a_i},\,\, \, \, r=\sum_{i=1}^s \dfrac{m_i}{na_i}.$
%\begin{eqnarray*}
%r_i = \dfrac{1}{a_i},\,\, \, \, r=\sum_{i=1}^s \dfrac{m_i}{na_i}.
%\end{eqnarray*} 
Thus, we obtain that the invariants of $x$ are given by
\begin{eqnarray*}
\rho^2 &=& \sum_{i}m_i(r_i - r)^2 = c_1,\\
C_i &=& -\rho^{-2}r_i\{e_i(r) - e_{i}(log\rho)(r - r_i) \} =0,\, \, \, \,  
i, j = 1,\ldots, n, \\
L_{ij} &=& \rho^{-2}\{\text{Hess}_{ij}(log\rho) - r_ir_je_i(log\rho)e_j(log\rho) + 
\dfrac{1}{2}\vert\nabla log\rho\vert^{2}\delta_{ij}\} = 0,\, \, \, \,  
i, j = 1,\ldots, n, \\
B_{ij} &=& b_i\delta_{ij},\, \, \, \,  
i, j = 1,\ldots, n, 
\end{eqnarray*}
where
\begin{equation}
b_i = \dfrac{r - r_i}{\sqrt{\sum_{j}m_j(r_j-r)^2}} = c_2,
\end{equation}
with $c_1,c_2\in\mathbb{R}$ are constants.
Therefore, we conclude that $x$ is an L-isotropic hypersurface
with $\lambda=0$.
\end{example}

\vspace{0.2in}

Considering the Laguerre immersion $\tau: U\mathbb{R}^{n+1}_{0}\rightarrow U\mathbb{R}^{n+1}$ we will obtain 
the hypersurface $x':\mathbb{R}^{n}\rightarrow\mathbb{R}^{n+1}$
defined by (\ref{imersao de Laguerre: relacao x e x'}),
corresponding to the hypersurface
$x$ described in the previous example.
From (\ref{imersao de Laguerre: relacao x e x'}), (\ref{componentes da x})
and (\ref{componentes da xi}) we obtain the L-isotropic hypersurface in Euclidean space
$x':\mathbb{R}^{n}\rightarrow\mathbb{R}^{n+1}$, given by
\begin{equation}\label{definicao da HILF}
x' = (0, u_1,\ldots, u_s) + \dfrac{\sum_{i=1}^{s}a_i\vert u_i\vert ^2}
{\sum_{i=1}^{s}a_i^{2}\vert u_i\vert ^2 + 1}(1, -a_1u_1, \ldots, -a_su_s),
\end{equation}   
and its unit normal $\xi':\mathbb{R}^{n}\rightarrow\mathbb{S}^{n}$ 
\begin{equation}\label{normal da xi linha exemplo}
\xi' = \dfrac{1}{\sum_{i=1}^{s}a_i^{2}\vert u_i\vert ^2 + 1}
\left(\sum_{i=1}^{s}a_i^{2}\vert u_i\vert ^2 - 1,2a_1u_1, \ldots, 2a_su_s  \right).
\end{equation}  
The example given in \eqref{definicao da HILF} 
was first obtained by Corro, Ferreira, and Tenenblat \cite{artigo do exemplo} in 1999.

A natural question is whether the family of hypersurfaces described in 
\eqref{definicao da HILF} is the only 
L-isotropic one, when $\lambda = 0$. 
A partial answer to this question will be presented in Theorem 11. Namelly, we will show that the answer is positive when the hypersurface admits a parametrization by lines of curvature.

To conclude this section, we introduce the following
definition of an L-isoparametric hypersurface
in $\mathbb{R}^{n+1}$. 
Subsequently, we will investigate
the classification of hypersurfaces 
that possess both properties simultaneously, 
being L-isotropic and L-isoparametric.

\begin{definition}
\rm
Let $x: M^n \rightarrow \mathbb{R}^{n+1}$ be a hypersurface without umbilic points and with nonvanishing principal curvatures. We say that $x$ is an \textit{L-isoparametric hypersurface} if the Laguerre form $\mathbb{C}$ is zero and the Laguerre principal curvatures, given in 
(\ref{curvaturas principais de Laguerre}), are all constant.
\end{definition}

Motivated by the work of Shu \cite{isop em Rn}, a natural question is to determine which hypersurfaces lie in the intersection of the classes of L-isotropic and L-isoparametric hypersurfaces.  
A hypersurface in $\mathbb{R}^{n+1}$ that has the property of being both L-isoparametric and L-isotropic simultaneously is given by Example \ref{exemplo da isotropica lambda =0}, where in this case $\lambda = 0$.  
We posed a new question: when $\lambda = 0$, is this example the only one in this intersection? The positve answer is given given in Proposition 7.

\begin{lemma}\label{LEMA IGC}(\cite{Li grande})
Consider $x: M^n \rightarrow \mathbb{R}^{n+1}$,  
$n \geq 3$, a hypersurface without umbilic points, with distinct nonzero principal curvatures.  
Suppose that $x$ is L-isoparametric.  
Then for every $i = 1, \ldots, n$, the following identity holds

\begin{equation}\label{IGC para tensor Bezinho}
    \displaystyle\sum_{j\notin [i]} 
    \dfrac{R_{ijij}}{(b_i-b_j)} =0,
\end{equation}
where $R_{ijij}$ is the sectional curvature on $M^n$  
with respect to the Laguerre metric $g$,  
$b_i$ are the Laguerre principal curvatures, and  
$[i]=\{l\, \, \, ;\,\, \, \, b_i = b_l\}$.
\end{lemma}

\begin{proposition}\label{isop mais isoto}
Let $x:M^n \rightarrow \mathbb{R}^{n+1}$,  
$n \geq 3$, be a hypersurface without umbilic points  
with distinct nonvanishing principal curvatures.  
If $x$ is both L-isotropic and L-isoparametric,  
then the eigenvalues of the tensor $\mathbb{L}$ are all zero,  
i.e., $\lambda = 0$.
\end{proposition}
\begin{proof}
%Let $x$ be L-isotropic. It follows from \eqref{relacao entre os invariantes L, B e C}  
%that the tensors $\mathbb{B}$ and $\mathbb{L}$ commute.  
%Thus, we can choose an orthonormal frame $\{E_1, \ldots, E_n\}$ such that,

%\begin{equation}
%B_{ij} = b_i\delta_{ij},\, \, \, \, \,  L_{ij} = \lambda\delta_{ij},
%\end{equation}
%where $b_i$, $i=1,\ldots,n$ are the Laguerre principal curvatures.
%\newline Since $x$ is L-isoparametric, by Lemma \ref{LEMA IGC} we have,
%\begin{equation}\label{cansei 4}
%     \displaystyle\sum_{j\notin [i]} 
%    \dfrac{R_{ijij}}{(b_i-b_j)} =0.
%\end{equation}
Since  
$x$ is L-isotropic, it follows from \eqref{Rijkl da isotropica correcaoo.} that,
$R_{ijkl} = 2\lambda(\delta_{jk}\delta_{il} -\delta_{jl}\delta_{ik}),\, \, \, 
i,j,k,l=1,\ldots,n.$
Hence,
\begin{equation}\label{cansei 3}
    R_{ijij} = 2\lambda,\, \, \, \, \, \, \, \, \,i\neq j=1,\ldots,n.
\end{equation}
Now, since $x$ is L-isoparametric, using Lemma \ref{LEMA IGC} and 
\eqref{cansei 3} we obtain,

\begin{eqnarray}\label{cancei correcao dia 23.}
\displaystyle\sum_{j\notin [i]} 
    \dfrac{2\lambda}{(b_i-b_j)} =0,  \, \, \, \, \, \, \, \, \,i=1,\ldots,n. 
\end{eqnarray}
%Assim, para $\lambda=0$ é suficiente que
%$$\displaystyle\sum_{j\notin [i]} 
%    \dfrac{1}{(b_i-b_j)} > 0,$$
%para algum $i=1,\ldots,n$.
Since the Laguerre principal curvatures, given in
\eqref{curvaturas principais de Laguerre} 
are constant for all $i=1,\ldots,n$ and
there are no umbilical points, we can reorder the orthonormal frame
$\{E_1,\ldots,E_n \}$ such that
for some $i_{0}=1,\ldots,n$

\begin{equation}
    b_{i_{0}}>b_{j}, \, \, \, \, \, \, \, \, \, \forall j\neq i_0=1,\ldots,n. 
\end{equation}
Then, for this index $i_0$ we have that
$\sum_{j\notin [i]} 
    \dfrac{1}{(b_{i_{0}}-b_j)} > 0.$
Therefore, it follows from \eqref{cancei correcao dia 23.}
that $\lambda=0$.
\end{proof}

\section{L-isotropic hypersurfaces in
\texorpdfstring{$\mathbb{R}^{n+1}$}{Rn+1} parameterized by lines of curvature }

Consider an L-isotropic hypersurface  $M^n\subset\mathbb{R}^{n+1}$, $n\geq 3$, parametrized by lines of curvature with $n$ distinct, nonvanishing principal curvatures. From \eqref{relacao entre os invariantes L, B e C} we can choose an orthonormal basis 
$\{E_{1},E_{2},\ldots,E_{n}\}$ with respect to the Laguerre metric $g$ satisfying \eqref{tensores B e L comutando ingles..}.
%\begin{equation}\label{B e L diagonais}
%B_{ij} = b_i\delta_{ij}\, \, \, \,  e\, 
%\, \, \, \, L_{ij} = \lambda\delta_{ij}\, \, \, 1\leq i,j\leq n.
%\end{equation}
Consider a coordinate system 
$(u_1,\ldots,u_n)$ around  $p\in M^n$ given by lines of curvature with respect to the Laguerre metric g. From now on, the index notation ``$,{i}$ '' denotes the partial derivative with respect to the coordinate $u_i$. Then
\begin{equation}\label{L metrica no sc}
    g\bigg(\dfrac{\partial}{\partial u_i},\dfrac{\partial}{\partial u_j}\bigg) = \bigg\langle dY\bigg(\dfrac{\partial}{\partial u_i}\bigg), dY\bigg(\dfrac{\partial}{\partial u_j}\bigg)\bigg\rangle = \langle Y_{,i} , Y_{,j} \rangle = \delta_{ij}g_{ii},
\end{equation}
In this coordinate system, the vector fields  $\{E_1, \ldots, E_n \}$ are given by 
\begin{equation}\label{campos Ei no sistema de coordenadas}
    E_i=\dfrac{1}{\sqrt{g_{ii}}}\dfrac{\partial}{\partial u_i},\, \, \, \, \,  i=1,\ldots,n.
\end{equation}

\begin{proposition}\label{lema 1 arrumando o sistema}
Consider $x:M^n\rightarrow\mathbb{R}^{n+1}$ an L-isotropic hypersurface with $n$ distinct, nonvanishing principal curvatures,
$Y:M^n\rightarrow C^{n+3}\subset
\mathbb{R}^{n+4}_2$ the Laguerre position vector of $x$ 
with the Laguerre metric $g=\langle dY,dY \rangle$ and
$\lambda$ the eigenvalue of the tensor $\mathbb{L}$.
Suppose that $M^n$ is connected and admits a parameterization by lines of curvature $(u_1,u_2,\ldots,u_n)\in 
U\subset\mathbb{R}^{n}$, where $U$ is open with respect to the Laguerre metric $g$. 
Then the following system of equations is satisfied by $N$, $Y$ and $\eta$, given in
\eqref{definicao da Y}, 
\eqref{definicao do N}, \eqref{aplicacao normal de laguerre}, respectively
\begin{equation}\label{sistema final em coordenadas}
\begin{cases}
N_{,i} = \lambda Y_{,i}, \\
\\
\dfrac{Y_{,ii}}{g_{ii}} - \dfrac{g_{ii,i}Y_{,i}}{2g_{ii}^2}  = \lambda Y + N + b_iP,  \\
\\
Y_{,ij} = \displaystyle\sum_{\substack{l=1,\\l \neq i \neq j}}^{n}\Gamma^{l}_{ij}\dfrac{\sqrt{g_{ii}}\sqrt{g_{jj}}}{\sqrt{g_{ll}}}Y_{,l},\, \, \, \forall i\neq j=1,\ldots, n,\\
\\
\eta_{,i} =b_{i}Y_{,i},
\end{cases}
\end{equation}
where $P=(1, -1, \Vec{0},0)$, $\Vec{0}\in\mathbb{R}^{n+1}$.
\end{proposition}
\begin{proof}
Substituting  
\eqref{tensores B e L comutando ingles..} and \eqref{campos Ei no sistema de coordenadas} into the structure equations 
\eqref{equacoes de estrutura menos gerais}, we obtain the following system of differential equations,
\begin{equation}\label{equacoes de estrutura caso L-isotropica em coordenadas}
\begin{cases}
\dfrac{1}{\sqrt{g_{ii}}}N_{,i} = \dfrac{1}{\sqrt{g_{ii}}}\lambda Y_{,i},\, \, \, \forall i=1,\ldots, n,\\
\\
\dfrac{Y_{,ii}}{g_{ii}} - \dfrac{g_{ii,i}Y_{,i}}{2g_{ii}^2} = \lambda Y + N + \displaystyle\sum _{k}\Gamma^{k}_{ii}\dfrac{1}{\sqrt{g_{kk}}}Y_{,k} + b_{i}P,\, \, \, \forall i=1,\ldots, n,\\
\\
\dfrac{Y_{,ij}}{\sqrt{g_{ii}}\sqrt{g_{jj}}} - \dfrac{g_{ii,j}Y_{,i}}{2g_{ii}\sqrt{g_{ii}}\sqrt{g_{jj}}} = \displaystyle\sum _{k}\Gamma^{k}_{ij}\dfrac{1}{\sqrt{g_{kk}}}Y_{,k},\, \, \, \forall i\neq j=1,\ldots, n,\\
\\
\dfrac{1}{\sqrt{g_{ii}}}\eta_{,i} = \dfrac{1}{\sqrt{g_{ii}}}b_{i}Y_{,i},\, \, \, \forall i=1,\ldots, n,
\end{cases}
\end{equation}
where $b_i$ are the principal curvatures of Laguerre and $\Gamma^{k}_{ij}$ are the Christoffel symbols 
in the metric $g$. We will prove that
\eqref{equacoes de estrutura caso L-isotropica em coordenadas} reduces to 
\eqref{sistema final em coordenadas}. First, we want to prove that 
$\Gamma^{i}_{ij} = 0$, for all $i,j=1,\ldots,n$. Furthermore, we will prove that
$\Gamma^{i}_{ij} = \Gamma^{i}_{ji} = 0$ is equivalent to $g_{ii,j} = 0$, for all $i\neq j=1,\ldots,n$.

As defined in 
\eqref{equacoes de estrutura menos gerais}, we have
$\Gamma ^{k}_{ij} = \omega_{ik}(E_j)$ for all $i, j, k=1,\ldots,n$. Hence we obtain that
\begin{equation}\label{gama iij}
\Gamma^{i}_{ij}=0,\, \, \, \, \, \, \, \, \,  \forall i, j=1,\ldots,n.
\end{equation}
Considering the inner product of the third equation in (\ref{equacoes de estrutura caso L-isotropica em coordenadas}) 
with $\dfrac{Y_{,i}}{\sqrt{g_{ii}}}$, $i\neq j$ and differentiating (\ref{L metrica no sc}) with respect to $u_j$, $i\neq j$ we have that 
$\Gamma^{i}_{ji} = \dfrac{g_{ii,j}}{2g_{ii}\sqrt{g_{jj}}}$ for all 
$i\neq j=1,\ldots,n.$ Then, it follows from (\ref{gama iij}) that the symmetry of the Christoffel symbols is equivalent to
%Agora, como $\Gamma^{i}_{ij}=0$, para todo $i\neq j=1,\ldots,n$, pela terceira equação de (\ref{equacoes de estrutura caso L-isotropica em coordenadas}) temos, 
%$$\dfrac{Y_{,ji}}{\sqrt{g_{ii}}\sqrt{g_{jj}}} - \dfrac{g_{jj,i}Y_{,j}}{2g_{jj}\sqrt{g_{ii}}\sqrt{g_{jj}}} = \displaystyle\sum _{k}\Gamma^{k}_{ji}\dfrac{1}{\sqrt{g_{kk}}}Y_{,k},\, \, \, \, \, \,  \forall i\neq j=1,\ldots, n.$$

%Uma observação importante a ser feita a respeito dos símbolos de Christoffel é com relação a simetria nos índices inferiores, isto é, $\Gamma^{i}_{ij} = \Gamma^{i}_{ji}$ se, e somente se, $g_{ii,j}=0$, 
%para todo $i\neq j=1,\ldots,n$.
%De fato, já vimos anteriormente que $\Gamma^{i}_{ij}=0$, para todo $i\neq j=1,\ldots,n$.
%Note que pela terceira equação de (\ref{equacoes de estrutura caso L-isotropica em coordenadas}) temos, 
%$$\dfrac{Y_{,ji}}{\sqrt{g_{ii}}\sqrt{g_{jj}}} - \dfrac{g_{jj,i}Y_{,j}}{2g_{jj}\sqrt{g_{ii}}\sqrt{g_{jj}}} = \displaystyle\sum _{k}\Gamma^{k}_{ji}\dfrac{1}{\sqrt{g_{kk}}}Y_{,k},\, \, \, \, \, \,  \forall i\neq j=1,\ldots, n.$$
%Considerando o produto interno dessa expressão com $\dfrac{Y_{,i}}{\sqrt{g_{ii}}}$, $i\neq j$ e diferenciando (\ref{L metrica no sc}) com relação a $u_j$, $i\neq j$ vemos que 
\begin{equation}\label{condicao dos simbolos}
    g_{ii,j}=0,\, \, \, \,  \forall i\neq j.
\end{equation}
Since the mixed partial derivatives of $Y$ commute, the remaining symbols satisfy 
$\Gamma^{l}_{ij} = \Gamma^{l}_{ji}$, for all $i\neq j\neq l.$ Therefore, it follows from (\ref{gama iij}) and 
(\ref{condicao dos simbolos}) that the system (\ref{equacoes de estrutura caso L-isotropica em coordenadas}) is given as in
(\ref{sistema final em coordenadas}).

\end{proof}

Note that the L-isotropic hypersurfaces in \( \mathbb{R}^{n+1} \) parameterized by curvature lines are determined through the integration of the system \eqref{sistema final em coordenadas}. Integrating this system and determining \( Y \), \( N \), and \( \eta \) provides us with an explicit parameterization for \( x \).
This will be our goal.
In what follows, by applying Frobenius' Theorem to the differential equations, we will obtain necessary and sufficient  conditions for the system \eqref{sistema final em coordenadas} to be integrable.

\begin{lemma}\label{lema da aplicacao do FroFro}
Under the same conditions as Proposition \ref{lema 1 arrumando o sistema}, the system 
\eqref{sistema final em coordenadas} is integrable if and only if

\begin{enumerate}
    \item $\lambda = 0$,
    \item $\Gamma^{j}_{li} = 0$ for all $i\neq j\neq l=1,\ldots,n$,
    \item  $b_i$ are constant for all $1,\ldots, n$.
\end{enumerate}
\end{lemma}
\begin{proof}
In order the system \eqref{sistema final em coordenadas}  to have a local solution, it is necessary and sufficient
to verify the condition of Frobenius' Theorem, i.e., the condition
$(Y_{,ii})_{,j}=(Y_{,ij})_{,i}, \, \, \forall i\neq j,$ must be satisfied.

Differentiating the second equation of \eqref{sistema final em coordenadas} with respect to $u_j, i\neq j$ we have,
\begin{equation}
 (Y_{,ii})_{,j} = g_{ii,j}(\lambda Y + N + b_iP) + g_{ii}(\lambda Y_{,j} + N_{,j} + b_{i,j}P) + \bigg(\dfrac{g_{ii,i}}{2g_{ii}}\bigg)_{,j}Y_{,i} + \dfrac{g_{ii,i}}{2g_{ii}}Y_{,ij}.
 \end{equation}
Using the first and third equations of \eqref{sistema final em coordenadas} and also the property that 
$g_{ii,j} = 0$ when $i\neq j$, we have
\begin{equation}\label{Yii derivada em j}
  (Y_{,ii})_{,j} = g_{ii}(2\lambda Y_{,j} + b_{i,j}P) + \bigg(\dfrac{g_{ii,i}}{2g_{ii}}\bigg)_{,j}Y_{,i} + \dfrac{g_{ii,i}}{2g_{ii}}\displaystyle\sum_{\substack{l=1,\\l \neq i \neq j}}^{n}\Gamma^{l}_{ij}\dfrac{\sqrt{g_{ii}}\sqrt{g_{jj}}}{\sqrt{g_{ll}}}Y_{,l},\, \, i\neq j.    
\end{equation}
Now, differentiating the equation for $Y_{,ij}, i\neq j$, given in \eqref{sistema final em coordenadas}, with respect to the coordinate  $u_i$ we obtain,
\begin{equation*}
  (Y_{,ij})_{,i} = \displaystyle\sum_{\substack{l=1,\\l \neq i \neq j}}^{n}\bigg(\Gamma^{l}_{ij}\dfrac{\sqrt{g_{ii}}\sqrt{g_{jj}}}{\sqrt{g_{ll}}}\bigg)_{,i}Y_{l} + \displaystyle\sum_{\substack{l=1,\\l \neq i \neq j}}^{n}\Gamma^{l}_{ij}\dfrac{\sqrt{g_{ii}}\sqrt{g_{jj}}}{\sqrt{g_{ll}}}Y_{,li}.
 \end{equation*}
Using the third equation of  \eqref{sistema final em coordenadas}, it follows that:
\begin{eqnarray*}
  (Y_{,ij})_{,i} &=& 
 \displaystyle\sum_{\substack{l=1,\\l \neq i \neq j}}^{n}\Bigg(\Gamma^{l}_{ij}\dfrac{\sqrt{g_{ii}}\sqrt{g_{jj}}}{\sqrt{g_{ll}}} \Bigg)_{,i}Y_{,l}  + \Bigg(\displaystyle\sum_{\substack{l=1,\\l \neq i \neq j}}^{n}\Gamma^{l}_{ij}\dfrac{\sqrt{g_{ii}}\sqrt{g_{jj}}}{\sqrt{g_{ll}}}\Bigg)\Bigg(\displaystyle\sum_{\substack{s=1,\\s \neq i \neq l}}^{n}\Gamma^{s}_{li}\dfrac{\sqrt{g_{ll}}\sqrt{g_{ii}}}{\sqrt{g_{ss}}}Y_{,s}\Bigg) \\ 
  &=& \displaystyle\sum_{\substack{l=1,\\l \neq i \neq j}}^{n}\Bigg(\Gamma^{l}_{ij}\dfrac{\sqrt{g_{ii}}\sqrt{g_{jj}}}{\sqrt{g_{ll}}} \Bigg)_{,i}Y_{,l} + g_{ii}\sqrt{g_{jj}}\displaystyle\sum_{\substack{l=1,\\l \neq i \neq j}}^{n}\Gamma^{l}_{ij}\Bigg[\Gamma^{j}_{li}\dfrac{Y_{,j}}{\sqrt{g_{jj}}}+\displaystyle\sum_{\substack{s=1,\\s \neq i \neq j\neq l}}^{n}\Gamma^{s}_{li}\dfrac{Y_{,s}}{\sqrt{g_{ss}}}\Bigg].\\ 
   %&=& \displaystyle\sum_{\substack{l=1,\\l \neq i \neq j}}^{n}\Bigg(\Gamma^{l}_{ij}\dfrac{\sqrt{g_{ii}}\sqrt{g_{jj}}}{\sqrt{g_{ll}}} \Bigg)_{,i}Y_{,l} + g_{ii}\displaystyle\sum_{\substack{l=1,\\l \neq i \neq j}}^{n}\Gamma^{l}_{ij}\Gamma^{j}_{li}Y_{,j}+ g_{ii}\sqrt{g_{jj}}\displaystyle\sum_{\substack{s=1,\\s \neq i \neq j\neq l}}^{n}\Gamma^{l}_{ij}\Gamma^{s}_{li}\dfrac{Y_{,s}}{\sqrt{g_{ss}}}.
 \end{eqnarray*}
Then, for $i\neq j$ we obtain,
\begin{equation}\label{Yij derivada em i}
  (Y_{,ij})_{,i} = \displaystyle\sum_{\substack{l=1,\\l \neq i \neq j}}^{n}\Bigg(\Gamma^{l}_{ij}\dfrac{\sqrt{g_{ii}}\sqrt{g_{jj}}}{\sqrt{g_{ll}}} \Bigg)_{,i}Y_{,l} + g_{ii}\displaystyle\sum_{\substack{l=1,\\l \neq i \neq j}}^{n}\Gamma^{l}_{ij}\Gamma^{j}_{li}Y_{,j}+ g_{ii}\sqrt{g_{jj}}\displaystyle\sum_{\substack{s=1,\\s \neq i \neq j\neq l}}^{n}\Gamma^{l}_{ij}\Gamma^{s}_{li}\dfrac{Y_{,s}}{\sqrt{g_{ss}}}.    
\end{equation}

Equating (\ref{Yii derivada em j}) and (\ref{Yij derivada em i})
since the frame is composed of linearly independent vectors, it follows that the coefficients of $Y_i$, $Y_j$, $Y_l$ $l\neq i\neq j$ and
$P$ are equal, i.e., for all $i\neq j$, we have
%\begin{eqnarray*}
%&&2g_{ii}\lambda Y_{,j} + g_{ii}b_{i,j}P + \bigg(\dfrac{g_{ii,i}}{2g_{ii}}\bigg)_{,j}Y_{,i} + \dfrac{g_{ii,i}}{2g_{ii}}\displaystyle\sum_{\substack{l=1,\\l \neq i \neq j}}^{n}\Gamma^{l}_{ij}\dfrac{\sqrt{g_{ii}}\sqrt{g_{jj}}}{\sqrt{g_{ll}}}Y_{,l}  \\ 
%&&=\displaystyle\sum_{\substack{l=1,\\l \neq i \neq j}}^{n}\Bigg(\Gamma^{l}_{ij}\dfrac{\sqrt{g_{ii}}\sqrt{g_{jj}}}{\sqrt{g_{ll}}} \Bigg)_{,i}Y_{,l} + g_{ii}\displaystyle\sum_{\substack{l=1,\\l \neq i \neq j}}^{n}\Gamma^{l}_{ij}\Gamma^{j}_{li}Y_{,j} + g_{ii}\sqrt{g_{jj}}\displaystyle\sum_{\substack{s=1,\\s \neq i \neq j\neq l}}^{n}\Gamma^{l}_{ij}\Gamma^{s}_{li}\dfrac{Y_{,s}}{\sqrt{g_{ss}}}
%\end{eqnarray*}
\begin{eqnarray}\label{condicao1}
&&\bigg(\dfrac{g_{ii,i}}{2g_{ii}}\bigg)_{,j} = 0,\, \, \, \, \forall i\neq j, \\ \label{condicao2}
&& 2\lambda = \displaystyle\sum_{\substack{l=1,\\l \neq i \neq j}}^{n}\Gamma^{l}_{ij}\Gamma^{j}_{li},\, \, \, \, \forall i\neq j, \\ \label{condicao3}
&&g_{ii}b_{i,j} = 0,\, \, \, \, \forall i\neq j, \\ \label{condicao4} \nonumber
&&\Gamma^{l}_{ij}\dfrac{g_{ii,i}\sqrt{g_{ii}}\sqrt{g_{jj}}}{2g_{ii}\sqrt{g_{ll}}} = \Bigg(\Gamma^{l}_{ij}\dfrac{\sqrt{g_{ii}}\sqrt{g_{jj}}}{\sqrt{g_{ll}}} \Bigg)_{,i} + g_{ii}\sqrt{g_{jj}}\displaystyle\sum_{\substack{s=1,\\s \neq i \neq j\neq l}}^{n}\Gamma^{s}_{ij}\Gamma^{l}_{si}\dfrac{1}{\sqrt{g_{ll}}},\, \forall l\neq i\neq j.\\
\end{eqnarray}
It follows from (\ref{condicao3}) that for each $i$, $b_i$ 
depends only on the variable $u_i$. Since $\sum _{i}B_{ii}=0$ we have that,
$b_{i,i} = 0.$ Therefore, $b_i$ are constants for all $i=1,\ldots,n$. 

Finally, let us look at equation (\ref{condicao2}).
Considering the inner product of the third equation in \eqref{sistema final em coordenadas} with 
$\dfrac{Y_{,s}}{\sqrt{g_{ss}}}, s\neq i\neq j$, we have
%Segue de \eqref{sistema final em coordenadas}
%que para quaisquer $i\neq j= 1\ldots,n$, $$Y_{,ij} = \displaystyle\sum_{\substack{l=1,\\l \neq i \neq j}}^{n}\Gamma^{l}_{ij}\dfrac{\sqrt{g_{ii}}\sqrt{g_{jj}}}{\sqrt{g_{ll}}}Y_{,l}.$$
%Considerando o produto interno dessa equação com $\dfrac{Y_{,s}}{\sqrt{g_{ss}}}, s\neq i\neq j $ temos,
\begin{equation}\label{simbolo misto}
     \Gamma^{s}_{ij} =  \dfrac{1}{\sqrt{g_{ii}}\sqrt{g_{jj}}\sqrt{g_{ss}}}\langle Y_{,ij}, Y_{,s}\rangle.
\end{equation} 
%Como $s\neq j$ sabemos que $\langle Y_{,j}, Y_{,s}\rangle = 0$, derivando com relação à coordenada $u_i$, 
%\begin{eqnarray}\label{mais uma derivada do PI}
% \langle Y_{,ji}, Y_{,s}\rangle  + \langle Y_{,j}, Y_{,si}\rangle = 0.
%\end{eqnarray}
Since the mixed derivatives of $Y$ commute,  
using  \eqref{L metrica no sc} and (\ref{simbolo misto}) we have,
\begin{eqnarray}\label{mais equação nos simbolos}
\Gamma^{s}_{ji} = \Gamma^{s}_{ij} &=& \dfrac{1}{\sqrt{g_{ii}}\sqrt{g_{jj}}\sqrt{g_{ss}}}\langle Y_{,ij}, Y_{,s}\rangle =- \dfrac{1}{\sqrt{g_{ii}}\sqrt{g_{jj}}\sqrt{g_{ss}}}\langle Y_{,j}, Y_{,si}\rangle.
\end{eqnarray}
Using again (\ref{simbolo misto}), with a reordering of the indices, it follows from  
\eqref{mais equação nos simbolos} that,
\begin{equation}\label{condicao importante dos simbolos - outro}
\Gamma^{s}_{ji} = - \Gamma^{j}_{si},\, \, \, \, \, \,  \forall i\neq j\neq s.
\end{equation}
Thus, from (\ref{condicao2}) and \eqref{condicao importante dos simbolos - outro} we have,
$2\lambda + \displaystyle\sum_{\substack{l=1,\\l \neq i \neq j}}^{n}(\Gamma^{j}_{li})^2 =0.$
%\begin{eqnarray}\label{condicao do lambda} \nonumber
%0&=& 2\lambda + \displaystyle\sum_{\substack{l=1,\\l \neq i \neq j}}^{n}(\Gamma^{j}_{li})^2 
%&=&2\lambda - \displaystyle\sum_{\substack{l=1,\\l \neq i \neq j}}^{n}\Gamma^{l}_{ji}\Gamma^{j}_{li} \\ \nonumber
%&=& 2\lambda - \displaystyle\sum_{\substack{l=1,\\l \neq i \neq j}}^{n}(-\Gamma^{j}_{li})\Gamma^{j}_{li} \\ 
%&=& 2\lambda + \displaystyle\sum_{\substack{l=1,\\l \neq i \neq j}}^{n}(\Gamma^{j}_{li})^2 .
%\end{eqnarray}
Since, by Proposition \ref{lambra maior ou igual a zero} we have
$\lambda \geq 0$ then, 
$\lambda =  \Gamma^{j}_{li} = 0,\, \, \, \, \forall i\neq j\neq l.$ 
Therefore, the system has a solution for L-isotropic hypersurfaces with $\lambda =0$. Moreover, 
$b_i$ is constant for all $i=1,\ldots,n$ and $\Gamma^{j}_{li} = 0$ for all $i\neq j\neq l=1,\ldots,n$. This concludes the proof.
\end{proof}

\begin{corollary}
Let  $x:M\rightarrow\mathbb{R}^{n+1}$
be an L-isotropic hypersurface with $\lambda >0$. Then $x$ does not admit a parameterization by lines of curvature.
\end{corollary}

\begin{lemma}\label{lema feito na segunda de carnaval}
Let $Y$, $N$ and $\eta$ vector fields in
$\mathbb{R}^{n+4}_2$
satisfying \eqref{sistema final em coordenadas}. Then, 
\begin{equation}
\begin{cases}\label{sistema que quero integrar final}
N_{,i} =0\\
\dfrac{Y_{,ii}}{g_{ii}} - \dfrac{g_{ii,i}Y_{,i}}{2g_{ii}^2}\ = N + b_iP \\
Y_{,ij} = 0, \, \, \, \, \, \forall i\neq j = 1,\ldots, n  \\
\eta_{,i} =b_{i}Y_{,i}
\end{cases}
\end{equation}
and
\begin{equation}\label{Y lema novo q nao sei mais o que colocar}
    Y = \Bigg(\displaystyle\sum _{k=1}^{n}F^{1}_{k}, -\displaystyle\sum _{k=1}^{n}F^{1}_{k}, \displaystyle\sum _{k=1}^{n}F^{3}_{k}, \displaystyle\sum _{k=1}^{n}F^{4}_{k},\ldots,
\displaystyle\sum _{k=1}^{n}F^{n+3}_{k},
\displaystyle\sum _{k=1}^{n}F^{3}_{k} + 1\Bigg),
\end{equation}
where $F^{l}_{k}$ is a differentiable function that depends only on $u_k$ for $l=1,\ldots, n+4$ and $k=1,\ldots, n$.
Moreover,
\begin{equation}\label{nova eq artigo}
g_{ii} = \langle Y_{,i},Y_{,i}\rangle = \displaystyle\sum 
_{t=1}^{n}\big(F^{t+3}_{i,i}\big)^2,\, \, \forall i= 1,\ldots, 
n.
\end{equation}
\end{lemma}
\begin{proof}
From Lemma \ref{lema da aplicacao do FroFro}, the system 
\eqref{sistema final em coordenadas} reduces to
\eqref{sistema que quero integrar final}. From the third equation of
\eqref{sistema que quero integrar final} it follows that
\begin{equation}
    Y = \displaystyle\sum _{k=1}^{n}F_k(u_k),
\end{equation}
where $F_k(u_k)= (F^{1}_{k}(u_k),F^{2}_{k}(u_k),\ldots,F^{n+4}_{k}(u_k))\in\mathbb{R}^{n+4}$ for all $k=1,\ldots,n$. Since $Y\in\mathbb{R}^{n+4}_2$, we will  
consider the coordinate functions of $Y$,
\begin{equation}\label{Coordenadas de Yl}
    Y^{l} = \displaystyle\sum _{k=1}^{n}F^{l}_{k}(u_k),
\end{equation}
where $F_{k}^{l}(u_k) \,  \,  \,  k=1,\ldots,n\, \, \, \text{and}\, \, \, l=1,\ldots,(n+4)$ denotes the $l$-th coordinate of the differentiable function $F_k$. With this notation, we will omit the variable $u_k$. Since $\langle Y,P \rangle = 0$, where $P=(1,-1,0,\vec{0},0)$, 
$\vec{0}\in\mathbb{R}^{n}$, we have that $\sum _{k=1}^{n}F^{1}_{k} + \sum _{k=1}^{n}F^{2}_{k} =0$, i.e.,
\begin{equation}\label{F1i=-F2i}
Y^{2}= -Y^{1}.
\end{equation}
Then,
\begin{eqnarray}\label{Y nao sei mais o que colocar}
    Y = \Bigg(\displaystyle\sum _{k=1}^{n}F^{1}_{k}, -\displaystyle\sum _{k=1}^{n}F^{1}_{k}, \displaystyle\sum _{k=1}^{n}F^{3}_{k}, \displaystyle\sum _{k=1}^{n}F^{4}_{k},\ldots,
    \displaystyle\sum _{k=1}^{n}F^{n+3}_{k},
\displaystyle\sum _{k=1}^{n}F^{(n+4)}_{k}\Bigg).
\end{eqnarray}
%Nosso objetivo é determinar explicitamente $Y$
%integrando o sistema (\ref{sistema que quero integrar final}). 
Note that the lightlike vector $N\in\mathbb{R}^{n+4}_2$ is constant, nonzero, and it is not a multiple of $P$.
In fact, from the first equation of (\ref{sistema que quero integrar final}) we have that $N=(N^{1},-N^{1},\ldots,N^{n+4})\in\mathbb{R}^{n+4}_2$ is constant, where $N^{2} = -N^{1}$, because $\langle N,P \rangle = 0$. Since $\langle N,Y \rangle = -1$, it follows that $N$ is nonzero. Moreover, if $N=N^{1}P$ we have a contradiction
$-1 = \langle N,Y \rangle = N^{1}\langle P,Y \rangle=0.$
Thus, $N=(N^{1},-N^{1},\ldots,N^{n+4})\in\mathbb{R}^{n+4}_2$ is a constant lightlike vector where there exists at least one index 
$m=3,\ldots,(n+4)$ such that $N^{m}\neq 0$.
We will choose, without loss of generality, $N=(0,0,1,\Vec{0},1)$, where 
$\Vec{0}\in\mathbb{R}^n$.

Thus, using the information that 
$\langle N,Y \rangle = -1$ we have, $Y^3 - Y^{n+4} = -1$, i.e.,
\begin{eqnarray}\label{F8 = 1+ F3}
\displaystyle\sum _{k=1}^{n}F^{(n+4)}_{k} &=& \displaystyle\sum _{k=1}^{n}F^{3}_{k} + 1.
\end{eqnarray}
Therefore, it follows from \eqref{Y nao sei mais o que colocar} and \eqref{F8 = 1+ F3} 
that the position vector $Y$ is given as in
\eqref{Y lema novo q nao sei mais o que colocar}.

Now, differentiating equation \eqref{F8 = 1+ F3}  with respect to $u_i$ we have,
\begin{equation}\label{F3i=F8i}
F^{n+4}_{i,i} = F^{3}_{i,i}, \, \, \, \forall i= 1,\ldots, n.
\end{equation}
Moreover, using (\ref{F1i=-F2i}) and (\ref{F3i=F8i}) it follows that \eqref{nova eq artigo} is satisfied. 

\end{proof}

We can now return to our main goal, which is to integrate the system (\ref{sistema que quero integrar final})
in order to obtain $Y, \eta\in\mathbb{R}^{n+4}_2$.
First, let us determine $Y$.
\begin{proposition}\label{correcao valter}
Consider $Y\in\mathbb{R}^{n+4}_2$
satisfying Lemma \ref{lema feito na segunda de carnaval}.
Then the differentiable functions
$F^{l}_{k}(u_k)$
$l=1,\ldots, n+4$ and $k=1,\ldots, n$ are given by
\begin{eqnarray}
&&F^{1}_{i} = \dfrac{h_{i}^{2}b_i}{2\big(c^{i+3}_{i}\big)^2} + h_i\beta^{1}_{i} + \gamma^{1}_{i}, \ \ \forall i=1,\ldots,n,\\
&&F^{s+3}_{i} = \dfrac{c^{s+3}_{i}}{c^{i+3}_{i}}h_i + \alpha^{s+3}_i, \ \ \forall i,s=1,\ldots,n, \\
&&F^{3}_{i} = \dfrac{h_{i}^{2}}{2\big(c^{i+3}_{i}\big)^2} + h_i\beta^{3}_{i} + \gamma^{3}_{i}, \ \ \forall i=1,\ldots,n,
\end{eqnarray}
where for all $i=1,\ldots,n$, $h_i(u_i)$ are differentiable functions  
that depend only on the variable $u_i$, $b_i$ are the Laguerre principal curvatures, and 
$c_{i}^{i+3}$, $\beta^{1}_{i}$, $\beta^{3}_{i}$,
$\gamma^{1}_{i}$, $\gamma^{3}_{i}$,
$\alpha^{s+3}_i$ are integration constants.  
Moreover,
\begin{eqnarray}\label{expressao gii em termo de hi}
g_{ii}=\dfrac{{h'}_{i}^{2}}{\big(c^{i+3}_{i}\big)^2},
\ \ \forall i=1,\ldots,n,
\end{eqnarray}
and 
\begin{eqnarray}\nonumber\label{Expressao final do Y em termo do hii}
Y^{l}
&=& \displaystyle\sum _{k=1}^{n}
\Bigg\{\Bigg[\dfrac{h_{k}^{2}}
{2\big(c^{k+3}_{k}\big)^2}b_k + h_{k}\beta^{1}_{k} + \gamma^{1}_{k} \Bigg] \Big(\delta_{l1} - \delta_{l2}\Big) \\ \nonumber
&+& \Bigg[ \dfrac{h_{k}^{2}}
{2\big(c^{k+3}_{k}\big)^2} + h_{k}\beta^{3}_{k}
+ \gamma^{3}_{k} \Bigg]\Big(\delta_{l3} + \delta_{l(n+4)}\Big)  \\
&+& \displaystyle\sum _{s=1}^{n}\Bigg(
\dfrac{c^{s+3}_{k}h_k}
{c^{k+3}_{k}} + \alpha^{s+3}_k\Bigg)\delta_{(s+3)l} \Bigg\}
+ 
\delta_{l(n+4)}, \ \ \forall l=1,\ldots,(n+4).
\end{eqnarray}
\end{proposition}
\begin{proof}
Since $N= (0,0,1,\Vec{0},1)$ and $P= 
(1,-1,0,0,\Vec{0})$, where $\Vec{0}\in\mathbb{R}^n$, we can rewrite the second equation of the system 
(\ref{sistema que quero integrar final}) for $Y$ in its coordinates $Y^l$ as follows:
\begin{eqnarray}
\dfrac{Y^{l}_{,ii}}{g_{ii}} - 
\dfrac{g_{ii,i}Y^{l}_{,i}}{2g_{ii}^2} = 
\delta_{l3} + \delta_{l(n+4)} + b_i(\delta_{l1} - \delta_{l2}),    
\end{eqnarray}
i.e.,
\begin{eqnarray}
\dfrac{1}{g_{ii}}\displaystyle\sum 
_{k=1}^{n}F^{l}_{k,ii} - 
\dfrac{g_{ii,i}}{2g_{ii}^2}\displaystyle\sum 
_{k=1}^{n}F^{l}_{k,i} = 
\delta_{l3} + \delta_{l(n+4)} + b_i(\delta_{l1} - \delta_{l2}).  
\end{eqnarray}
Since $F_k^l$ depends only on $u_k$, we have 
\begin{eqnarray}\label{segunda eq com os Fl}
\dfrac{F^{l}_{i,ii}}{g_{ii}} - 
\dfrac{g_{ii,i}F^{l}_{i,i}}{2g_{ii}^2} &=& 
\delta_{l3} + \delta_{l(n+4)} + b_i(\delta_{l1} - \delta_{l2}),\, \, \, \text{where}\, \, \, 
\left\{
\begin{array}{ll}
    l = 1,\ldots, (n+4),  \\
    i= 1,\ldots, n.
\end{array}
\right.
\end{eqnarray}
For $l=s+3$ and $s= 1,\ldots, n$ in \eqref{segunda eq com os Fl}, we get
$\dfrac{F^{s+3}_{i,ii}}{g_{ii}} - 
\dfrac{g_{ii,i}F^{s+3}_{i,i}}{2g_{ii}^2}  = 0$. Therefore, 
\begin{equation}\label{segunda correcao de 2.3...}
 \Bigg(\dfrac{F^{s+3}_{i,i}}{\sqrt{g_{ii}}}\Bigg)_{,i} = 0,  \ \ \forall i, s= 1,\ldots, n.   
\end{equation}
Integrating \eqref{segunda correcao de 2.3...} with respect to the variable $u_i$,
%\begin{equation}\label{equacao das F'(j+3)i}
%\dfrac{F^{s+3}_{i,i}}{\sqrt{g_{ii}}} = c^{s+3}_{i}
%, \,\, \,\forall i, s= 1,\ldots, n,
%\end{equation}
%from which 
\begin{equation}\label{equacao das F'(j+3)icorrecao2.3}
\big(c^{s+3}_{i}\big)^2g_{ii} - \big(F^{s+3}_{i,i}\big)^2  = 0, \,\, \,\forall i, s= 1,\ldots, n,
\end{equation}
where $c^{s+3}_{i}$ are constants of integration.
Note that from \eqref{equacao das F'(j+3)icorrecao2.3}, 
choosing the constants $c_{i}^{s+3}$ for 
$i, s= 1,\ldots, n$ means choosing $Y^{s+3}_{,i}$,
which are initial conditions for the system of partial differential equations
\eqref{sistema que quero integrar final}.

Denote
\begin{equation}\label{V=F correcao nova de 2.3}
    V^{s+3}_{i} = \big(F^{s+3}_{i,i}\big)^2.
\end{equation} 
Then, it follows from  (\ref{nova eq artigo}), \eqref{equacao das F'(j+3)icorrecao2.3} and \eqref{V=F correcao nova de 2.3} that,
\begin{eqnarray}\label{equação dos V}
%\big(c^{s+3}_{i}\big)^2\displaystyle\sum 
%_{t=1}^{n}V^{t+3}_{i} - V^{s+3}_{i} &=& 0 \\
\bigg[\big(c^{s+3}_{i}\big)^2 - 1\bigg]V^{s+3}_{i} + 
\big(c^{s+3}_{i}\big)^2\displaystyle\sum_{\substack{t=1,\\t \neq s}}^{n}V^{t+3}_{i} 
=0, \,\, \,\forall i, s= 1,\ldots, n.
\end{eqnarray}
\newline For each fixed  $i=1,\ldots, n$ we can rewrite 
\eqref{equação dos V} in matrix form as follows,
$(\textit{C}_{i})_{n\times n}\textit{V}_{n\times 1}= 
\textit{0}_{n\times 1},$ where
\begin{equation}\nonumber
 (\textit{C}_{i})_{n\times n} =  \left(\begin{array}{ccccc}
\big(c^{4}_{i}\big)^2-1 & \big(c^{4}_{i}\big)^2  & 
\ldots & \big(c^{4}_{i}\big)^2 \\
\big(c^{5}_{i}\big)^2& \big(c^{5}_{i}\big)^2-1 & 
\ldots & \big(c^{5}_{i}\big)^2 \\
\vdots & \vdots & \ddots & \vdots \\
\big(c^{n+3}_{i}\big)^2 & \big(c^{n+3}_{i}\big)^2 & 
\ldots & 
\big(c^{n+3}_{i}\big)^2-1\\
\end{array}\right)_{n\times n}
\textit{V}_{n\times 1} = \left(\begin{array}{c} 
V^{4}_{i}\\ V^{5}_{i}\\ \vdots \\ V^{n+3}_{i}
\end{array}\right)_{n\times 1} 
\textit{0}_{n\times 1} =\left(\begin{array}{c} 0\\0\\ \vdots \\ 0 \end{array} \right)_{n\times 1}.
\end{equation}
Note that $det(\textit{C}_{i})= 1 - \displaystyle\sum 
_{t=1}^{n}\big(c^{t+3}_{i}\big)^2$.
Since we are looking for nontrivial solutions of this system of equations, the determinant of the matrix $\textit{C}_{i}$ must be zero, i.e., 
\begin{equation}\label{Condicao 1 nas constantes c}
\displaystyle\sum_{t=1}^{n}\big(c^{t+3}_{i}\big)^2 = 1
\ \ \ \ \ \forall i=1,\ldots,n.
\end{equation}
For each fixed  $i=1,\ldots,n$ it is sufficient that at least one element of the sum 
\eqref{Condicao 1 nas constantes c} is nonzero. I.e., for each $i$ there exists 
a $t$ such that $c^{t+3}_i\neq 0$.
Without loss of generality, assume that
$\big(c^{i+3}_{i}\big)^2\neq 0$, for all  
$i=1,\ldots,n.$ 
Note that the derivative of the position vector $Y$ in the direction 
$u_i$ is the tangent vector to the coordinate curves at a point of $M$. 
Since these derivatives are given in terms of the constants $c's$, as in 
\eqref{equacao das F'(j+3)icorrecao2.3}, when we choose these constants we are determining the tangent space of $M$ at that point.

Observe that in \eqref{equação dos V} we have a linear system with $n^2$ equations and $n^2$ functions. 
We will determine the functions $V$ by means of substitutions.
Fix an index \eqref{equação dos V} in
$i=1,\ldots,n $ and consider $s=i$,
$$ \bigg[\big(c^{i+3}_{i}\big)^2 - 1\bigg]V^{i+3}_{i} + 
\big(c^{i+3}_{i}\big)^2\displaystyle\sum_{\substack{t=1,\\t \neq i}}^{n}V^{t+3}_{i} =0.$$
Since $\big(c^{i+3}_{i}\big)^2\neq 0$, we obtain
\begin{equation}\label{i fixo e s=i no sistema}
    V^{n+3}_{i} = \dfrac{[1-\big(c^{i+3}_{i}\big)^2]}{\big(c^{i+3}_{i}\big)^2}V^{i+3}_i - 
    \displaystyle\sum_{\substack{t=1,\\t \neq i}}^{n-1}V^{t+3}_{i}.
\end{equation}
Now, considering other fixed indices,
$s\neq i$ and $s\neq n$ in \eqref{equação dos V}, we have 
\begin{eqnarray}\label{i fixo e sdiferentei no sistema}
%\bigg[\big(c^{s+3}_{i}\big)^2 - 1\bigg]V^{s+3}_{i} + 
%\big(c^{s+3}_{i}\big)^2\displaystyle\sum_{\substack{t=1,\\t \neq s}}^{n}V^{t+3}_{i} &=&0 \\
 \bigg[\big(c^{s+3}_{i}\big)^2 - 1\bigg]V^{s+3}_{i} + 
\big(c^{s+3}_{i}\big)^2\bigg(\displaystyle\sum_{\substack{t=1,\\t \neq s}}^{n-1}V^{t+3}_{i} + V^{n+3}_{i}\bigg) =0 \ \ \  s\neq n.
\end{eqnarray}
Substituting \eqref{i fixo e s=i no sistema} into
\eqref{i fixo e sdiferentei no sistema}, it follows from a straightforward 
computation that,

%\begin{eqnarray}\nonumber
%&&\bigg[\big(c^{s+3}_{i}\big)^2 - 1\bigg]V^{s+3}_{i} + 
%\big(c^{s+3}_{i}\big)^2\bigg(\displaystyle\sum_{\substack{t=1,\\t \neq s}}^{n-1}V^{t+3}_{i} + \dfrac{[1-\big(c^{i+3}_{i}\big)^2]}{\big(c^{i+3}_{i}\big)^2}V^{i+3}_i - 
%    \displaystyle\sum_{\substack{t=1,\\t \neq i}}^{n-1}V^{t+3}_{i}\bigg) =0\\ \nonumber 
%&&\bigg[\big(c^{s+3}_{i}\big)^2 - 1\bigg]V^{s+3}_{i} + 
%\big(c^{s+3}_{i}\big)^2\bigg(V^{i+3}_{i} + \displaystyle\sum_{\substack{t=1,\\t \neq s,
%t \neq i}}^{n-1}V^{t+3}_{i} +  \dfrac{[1-\big(c^{i+3}_{i}\big)^2]}{\big(c^{i+3}_{i}\big)^2}V^{i+3}_i - V^{s+3}_{i} 
%   - \displaystyle\sum_{\substack{t=1,\\t \neq s,
%t \neq i}}^{n-1}V^{t+3}_{i}\bigg) =0\\ \nonumber 
%&&-V^{s+3}_{i} + 
%\bigg(\dfrac{\big(c^{s+3}_{i}\big)^2[1-\big(c^{i+3}_{i}\big)^2]}{\big(c^{i+3}_{i}\big)^2} + \big(c^{s+3}_{i}\big)^2\bigg)V^{i+3}_i =0.
%\end{eqnarray}
%Logo, 
\begin{equation}\label{expressao Vs+3 s diferente de i}
V^{s+3}_{i} = \dfrac{\big(c^{s+3}_{i}\big)^2}{\big(c^{i+3}_{i}\big)^2}V^{i+3}_i\ \ \  s =1,\ldots,(n-1),\, \, s\neq i,\, \, \, i=1,\ldots,n.
\end{equation}

Finally, using the expressions
\eqref{Condicao 1 nas constantes c}, \eqref{i fixo e s=i no sistema} and
\eqref{expressao Vs+3 s diferente de i} we obtain,
%\begin{eqnarray}\nonumber
%V^{n+3}_{i}&=& \dfrac{[1-\big(c^{i+3}_{i}\big)^2]}{\big(c^{i+3}_{i}\big)^2}V^{i+3}_i %- 
%\displaystyle\sum_{\substack{t=1,\\
%t \neq i}}^{n-1}\dfrac{\big(c^{t+3}_{i}\big)^2}{\big(c^{i+3}_{i}\big)^2}V^{i+3}_i\\ \nonumber
% &=& \dfrac{V^{i+3}_i}{\big(c^{i+3}_{i}\big)^2}\bigg(1- \big(c^{i+3}_{i}\big)^2 - \displaystyle\displaystyle\sum_{\substack{t=1,\\
%t \neq i}}^{n-1}\big(c^{t+3}_{i}\big)^2\bigg)\\ \nonumber
 %&=& \dfrac{V^{i+3}_i}{\big(c^{i+3}_{i}\big)^2}\bigg(1-\displaystyle\sum_{t=1}^{n-%1}\big(c^{t+3}_{i}\big)^2\bigg).
%\end{eqnarray}
%Utilizando \eqref{Condicao 1 nas constantes c} obtemos,
\begin{equation}\label{expressao Vs+3 s igual de i
correcao 23}
V^{n+3}_{i} = \dfrac{\big(c^{n+3}_{i}\big)^2}{\big(c^{i+3}_{i}\big)^2}V^{i+3}_i, \ \ \ i=1,\ldots,n.
\end{equation}
Therefore, it follows from \eqref{expressao Vs+3 s diferente de i} and \eqref{expressao Vs+3 s 
igual de i correcao 23} that all functions $V^{s+3}_{i}$, 
for all $i,s \in \{1, \ldots, n\}$ with $s \neq i$, 
are expressed in terms of a differentiable function $V^{i+3}_i$ that depends only on the variable $u_i$.

Consider $V^{i+3}_i = h_i'(u_i)^2$, for all $i = 1, \ldots, n$. From now on, in order to simplify our notation, we will omit the variable in the functions 
$h_i'(u_i)$, understanding that $h_i$, and its derivatives, are differentiable functions depending only on the variable $u_i$.

It follows from \eqref{V=F correcao nova de 2.3}, \eqref{expressao Vs+3 s diferente de i}, and \eqref{expressao Vs+3 s 
igual de i correcao 23} that for all 
$i,s = 1, \ldots, n$ we have $\big(F^{s+3}_{i,i}\big)^2= V^{s+3}_{i} = 
  \dfrac{\big(c^{s+3}_{i}\big)^2}
  {\big(c^{i+3}_{i}\big)^2}{h'}_{i}^{2}.$

%\begin{equation}\label{expressao final V}
 %   \big(F^{s+3}_{i,i}\big)^2= V^{s+3}_{i} = 
%  \dfrac{\big(c^{s+3}_{i}\big)^2}
 % {\big(c^{i+3}_{i}\big)^2}{h'}_{i}^{2}.
%\end{equation}
\noindent Therefore,
\begin{equation}\label{expressao final F}
    F^{s+3}_{i,i} = \dfrac{c^{s+3}_{i}}{c^{i+3}_{i}}h'_i, \ \ \forall i,s=1,\ldots,n.
\end{equation}
Integrating \eqref{expressao final F} with respect to $u_i$, we obtain
\begin{equation}\label{expressao final F integrada}
    F^{s+3}_{i} = \dfrac{c^{s+3}_{i}}{c^{i+3}_{i}}h_i + \alpha^{s+3}_i, \ \ \forall i,s=1,\ldots,n,
\end{equation}
where $\alpha^{s+3}_i$ are constants of integration.

Once these functions are obtained, we can rewrite the components of the 
metric $g_{ii}$ in terms of the derivatives of the functions $h_i$ and 
the constants $c$. Making use of 
\eqref{nova eq artigo},  \eqref{Condicao 1 nas constantes c} and
\eqref{expressao final F} we obtain,
\begin{eqnarray*}
g_{ii}&=& \displaystyle\sum 
_{t=1}^{n}\big(F^{t+3}_{i,i}\big)^2 = \displaystyle\sum 
_{t=1}^{n} \dfrac{\big(c^{t+3}_{i}\big)^2}
  {\big(c^{i+3}_{i}\big)^2}{h'}_{i}^{2},
\ \ \forall i=1,\ldots,n.
\end{eqnarray*}
Therefore, we have proved
\eqref{expressao gii em termo de hi}. 
%Derivando a expressão de $g_{ii}$
%dada em \eqref{expressao gii em termo de hi} 
%com relação a $u_i$ temos,
%\begin{equation}\label{expressao gii,i em termo de hi}
%g_{ii,i}=\dfrac{2h'_{i}h''_{i}}{\big(c^{i+3}_{i}\big)^2},
%\ \ \forall i=1,\ldots,n.    
%\end{equation}

In what follows, we will determine the remaining functions needed to express 
the position vector $Y$, namely the functions $F^1_i$ and $F^3_i$, for all 
$i = 1, \ldots, n$. Choosing $l = 1$ in \eqref{segunda eq com os Fl}, we have
\begin{equation*}
  \dfrac{F^{1}_{i,ii}}{g_{ii}} - 
\dfrac{g_{ii,i}F^{1}_{i,i}}{2g_{ii}^2} = b_i, \ \ \forall i=1,\ldots,n. 
\end{equation*}
Substituting expression \eqref{expressao gii em termo de hi} and its derivatives 
with respect to $u_i$ into the previous equation, we obtain the following 
ordinary differential equation, 
\begin{eqnarray}\nonumber
 \dfrac{F^{1}_{i,ii}\big(c^{i+3}_{i}\big)^2}{{h'}_{i}^{2}} - 
\dfrac{h''_{i}F^{1}_{i,i}\big(c^{i+3}_{i}\big)^2}{{h'}_{i}^{3}} = b_i.
\end{eqnarray}
Multiplying both sides by 
$\dfrac{{h'}_{i}^{2}}{\big(c^{i+3}_{i}\big)^2}$, 
and recalling that $\big(c^{i+3}_{i}\big)^2 \neq 0$ for all $i = 1, \ldots, n$, 
we obtain

\begin{equation}\label{EDO F1}
 F^{1}_{i,ii} - 
\dfrac{h''_{i}F^{1}_{i,i}}{{h'}_{i}} =\dfrac{b_i{h'}_{i}^{2}}{\big(c^{i+3}_{i}\big)^2}, \ \ \forall i=1,\ldots,n.   
\end{equation}
Integrating equation \eqref{EDO F1}, it follows that
\begin{equation}\label{F1 final}
 F^{1}_{i} = \dfrac{h_{i}^{2}b_i}{2\big(c^{i+3}_{i}\big)^2} + h_i\beta^{1}_{i} + \gamma^{1}_{i}, \ \ \forall i=1,\ldots,n,   
\end{equation}
where $\beta^{1}_{i}$ and $\gamma^{1}_{i}$ are constants of integration 
for all $i=1,\ldots,n$.

Finally, considering $l = 3$ in \eqref{segunda eq com os Fl}, we have
$\dfrac{F^{3}_{i,ii}}{g_{ii}} - 
\dfrac{g_{ii,i}F^{3}_{i,i}}{2g_{ii}^2} = 1$, for all $i=1,\ldots,n.$
%\begin{equation*}
%  \dfrac{F^{3}_{i,ii}}{g_{ii}} - 
%\dfrac{g_{ii,i}F^{3}_{i,i}}{2g_{ii}^2} = 1, \ \ \forall i=1,\ldots,n. 
%\end{equation*}
Similarly to the previous calculations, we obtain
%\begin{equation}\label{EDO F3}
% F^{3}_{i,ii} - 
%\dfrac{h''_{i}F^{3}_{i,i}}{{h'}_{i}} =\dfrac{{h'}_{i}^{2}}{\big(c^{i+3}_{i}\big)^2}, \ \ \forall i=1,\ldots,n.   
%\end{equation}
%Via integração de \eqref{EDO F3} obtemos,
\begin{equation}\label{F3 final}
 F^{3}_{i} = \dfrac{h_{i}^{2}}{2\big(c^{i+3}_{i}\big)^2} + h_i\beta^{3}_{i} + \gamma^{3}_{i}, \ \ \forall i=1,\ldots,n,   
\end{equation}
where $\beta^{3}_{i}$ and $\gamma^{3}_{i}$ are constants of integration for 
all $i=1,\ldots,n$.
%Detalhes para obter as expressões 
%\eqref{F1 final} e \eqref{F3 final} podem ser 
%vistas no Apêndice C.

Substituting the equations \eqref{expressao final F integrada}, 
\eqref{F1 final}, and \eqref{F3 final} into 
\eqref{Y lema novo q nao sei mais o que colocar}, 
we obtain the first part of the proposition, which implies that each coordinate function of the position vector $Y$ is given by
\eqref{Expressao final do Y em termo do hii}.
\end{proof}
Now, our goal is to determine the vector field
$\eta\in C^{n+3}\subset\mathbb{R}^{n+4}_{2}$.
\begin{proposition}
Consider $Y$, $N$, and $\eta$ as vector fields in 
$\mathbb{R}^{n+4}_2$ satisfying (\ref{sistema que quero integrar final}). 
Let $Y$ be given as in \eqref{Expressao final do Y em termo do hii} in terms of $h_k(u_k)$, the Laguerre principal curvatures $b_k$ and the constants $c_{k}^{k+3}$, $\beta^{1}_{k}$, and $\beta^{3}_{k}$.
Then the coordinate functions of $\eta$, denoted by $\eta^{l}$ for 
$l = 1, \ldots, (n+4)$, are given by:

\begin{eqnarray}\nonumber\label{Expressao final do eta em termo do hii}
\eta^{l}&=& \displaystyle\sum _{k=1}^{n}
\Bigg\{\Bigg[\dfrac{h_{k}^{2}b_{k}^2}{2\big(c^{k+3}_{k}\big)^2} + h_{k}\beta^{1}_{k}b_{k}\Bigg]\Big(\delta_{l1} - \delta_{l2}\Big)  
+ \Bigg[\dfrac{h_{k}^{2}b_{k}}{2\big(c^{k+3}_{k}\big)^2} + h_{k}\beta^{3}_{k}b_{k} \Bigg](\delta_{l3} + \delta_{l(n+4)}) \\
&+&\displaystyle\sum _{s=1}^{n}
\dfrac{c^{s+3}_{k}h_{k}b_{k}}{c^{k+3}_k}\delta_{l(s+3)}
\Bigg\} + \psi^l, \ \ \forall l=1,\ldots,(n+4),
\end{eqnarray}
where $\psi^l$ are constants 
of integration, such that
 $\psi^3 = \psi^{n+4}$.
\end{proposition}
\begin{proof}
Since $\langle \eta, N \rangle = 0$, we have
 \begin{eqnarray}\label{eta3=etan+4}
\eta^3 - \eta^{(n+4)}  &=&0,
\end{eqnarray}
where $\eta^{l}$, for all $l=1,\ldots,(n+4)$ are
the coordinate functions of $\eta$.

%In what follows, we will work with the last equation of the system
%\eqref{sistema que quero integrar final}, given by
%\begin{equation}\label{lembrando eq. eta}
%    \eta_{,i}=b_iY_{,i},\ \ \forall i=1,\ldots,n.
%\end{equation}
Using the last equation of \eqref{sistema que quero integrar final}, since 
$b_i$ are constants and $Y_{,ij} = 0$ for all $i \neq j = 1, \ldots, n$, we note that 
$\eta$ is also a function of separated variables. 
I.e., for each $l = 1, \ldots, (n+4)$, the coordinate functions of $\eta$ are 
written as a sum of differentiable functions, each depending on a single variable.
Rewriting the last equation of \eqref{sistema que quero integrar final} in 
matrix form, we have

\begin{eqnarray*}
\left(\begin{array}{cccc}
\eta_{,1}^{1} & \eta_{,1}^{2}  & 
\ldots & \eta_{,1}^{(n+4)}\\
\eta_{,2}^{1} & \eta_{,2}^{2}  & 
\ldots & \eta_{,2}^{(n+4)}\\
\vdots & \vdots & \ddots & \vdots \\
\eta_{,n}^{1} & \eta_{,n}^{2}  & 
\ldots & \eta_{,n}^{(n+4)}
\end{array}\right)_{n\times(n+4)}=
\left(\begin{array}{cccc} 
b_{1} & 0 & \ldots & 0 \\ 
0 & b_{2} & \ldots & 0 \\ 
\vdots & \vdots & \ddots & \vdots \\
0 & 0 & \ldots & b_{n} 
\end{array}\right)_{n\times n} 
 \left(\begin{array}{cccc}
Y_{,1}^{1} & Y_{,1}^{2}  & 
\ldots & Y_{,1}^{(n+4)}\\
Y_{,2}^{1} & Y_{,2}^{2}  & 
\ldots & Y_{,2}^{(n+4)}\\
\vdots & \vdots & \ddots & \vdots \\
Y_{,n}^{1} & Y_{,n}^{2}  & 
\ldots & Y_{,n}^{(n+4)}
\end{array}\right)_{n\times(n+4)}
\end{eqnarray*}
where, similarly to the proof of Proposition \ref{correcao valter} for the position vector $Y$, 
$\eta^l_{,i},\ \ i=1,\ldots,n,\ \ l=1,\ldots,(n+4)$ 
denotes the derivative of the $l$-th coordinate function 
of $\eta$ with respect to the variable $u_i$.

Differentiating \eqref{Expressao final do Y em termo do 
hii} with respect to $u_i$, we obtain for all $i=1,\ldots,n$ and 
$l=1,\ldots,(n+4)$ that

\begin{eqnarray*}
 Y^{l}_{,i}=  \Bigg[\dfrac{h_{i}h'_{i}}
{\big(c^{i+3}_{i}\big)^2}b_i + h'_{i}\beta^{1}_{i}\Bigg] \Big(\delta_{l1} - \delta_{l2}\Big) 
+ \Bigg[ \dfrac{h_{i}h'_{i}}
{\big(c^{i+3}_{i}\big)^2} + h'_{i}\beta^{3}_{i}
\Bigg]\Big(\delta_{l3} + \delta_{l(n+4)}\Big) 
+ \displaystyle\sum _{s=1}^{n}\Bigg(
\dfrac{c^{s+3}_{i}h'_i}
{c^{i+3}_{i}} \Bigg)\delta_{l(s+3)}.
\end{eqnarray*}
Since $\eta_{,i} = b_i Y_{,i}$, we have that 
$\eta^l_{,i} = b_i Y^l_{,i}$, i.e.,
\begin{eqnarray*}
\eta^{l}_{,i}=  \Bigg[\dfrac{h_{i}h'_{i}}
{\big(c^{i+3}_{i}\big)^2}b_i^2 + h'_{i}\beta^{1}_{i}b_i\Bigg] \Big(\delta_{l1} - \delta_{l2}\Big) 
+ \Bigg[ \dfrac{h_{i}h'_{i}b_i}
{\big(c^{i+3}_{i}\big)^2} + h'_{i}\beta^{3}_{i}b_i
\Bigg]\Big(\delta_{l3} + \delta_{l(n+4)}\Big) 
+ \displaystyle\sum _{s=1}^{n}\Bigg(
\dfrac{c^{s+3}_{i}h'_ib_i}
{c^{i+3}_{i}} \Bigg)\delta_{l(s+3)}.
\end{eqnarray*}
Integrating the previous expression and using \eqref{eta3=etan+4}, we obtain,
\begin{eqnarray}\nonumber
\eta^{l}
&=& \displaystyle\sum _{k=1}^{n}
\Bigg\{\Bigg[\dfrac{h_{k}^{2}b_{k}^2}{2\big(c^{k+3}_{k}\big)^2} + h_{k}\beta^{1}_{k}b_{k}\Bigg]\Big(\delta_{l1} - \delta_{l2}\Big)  
+ \Bigg[\dfrac{h_{k}^{2}b_{k}}{2\big(c^{k+3}_{k}\big)^2} + h_{k}\beta^{3}_{k}b_{k} \Bigg](\delta_{l3} + \delta_{l(n+4)}) \\ \nonumber
&+&\displaystyle\sum _{s=1}^{n}
\dfrac{c^{s+3}_{k}h_{k}b_{k}}{c^{k+3}_k}\delta_{l(s+3)}
\Bigg\} + \psi^l, \ \ \forall l=1,\ldots,(n+4),
\end{eqnarray}
where $\psi^l\in\mathbb{R}$ is the constant of integration such that 
$\psi^3 = \psi^{n+4}$, since \eqref{eta3=etan+4} is satisfied.
\end{proof}

In order to simplify the expressions, let us consider the following change 
of variables
\begin{equation}\label{mudança de coordenadas 1}
v_i=\dfrac{h_i}{\sqrt{2}c^{i+3}_{i}}+
\dfrac{c^{i+3}_{i}\beta^{3}_{i}}
{\sqrt{2}}, 
\ \ \ \forall i=1,\ldots,n.
\end{equation}
Note that performing this change of variable means that 
we are reparametrizing each line of curvature separately.
Moreover, considering \eqref{mudança de coordenadas 1}
%Em um primeiro momento essa modificação pode parecer 
%artificial. No entanto, completando quadrados
%na expressão de $Y^3$ dada em 
%\eqref{Expressao final do Y 
%em termo do hii} essa mudança aparece de maneira 
%natural.
the expressions
\eqref{Expressao final do Y em termo do hii} and
\eqref{Expressao final do eta em termo do hii}
for $Y$ and $\eta$, respectively, can be rewritten as follows,

\begin{eqnarray}\label{Expressao depois da mudança do Y}\nonumber
Y^{l}&=&
\displaystyle\sum _{k=1}^{n}
\Bigg\{\Bigg[v_{k}^{2}b_k  + 
\sqrt{2}v_{k}c^{k+3}_{k}(\beta^{1}_{k} - 
\beta^{3}_{k}b_k) + 
\big(c^{k+3}_{k}\big)^2\beta^{3}_{k}
\Big(\dfrac{\beta^{3}_{k}b_k}{2} -
\beta^{1}_{k}\Big) +  \gamma^{1}_{k} \Bigg]
\Big(\delta_{l1} - \delta_{l2}\Big) \\
&+& \Bigg[ v_k^2 - 
\dfrac{\big(c^{k+3}_{k}\beta^{3}_{k}\big)^2}{2} 
+ \gamma^{3}_{k} \Bigg]\Big(
\delta_{l3} + \delta_{l(n+4)}\Big) \\  \nonumber
&+&
\displaystyle\sum _{s=1}^{n}\Big\{
\sqrt{2}v_{k}c^{s+3}_{k} - c^{s+3}_{k}c^{k+3}_{k}\beta^{3}_{k} +
\alpha^{s+3}_k\Big\}\delta_{l(s+3)}\Bigg\}
 + 
\delta_{l(n+4)}, \ \ \forall l=1,\ldots,(n+4),
\end{eqnarray}
and 
\begin{eqnarray}\nonumber
\label{Expressao depois da mudança no eta}
\eta^{l}&=& 
 \displaystyle\sum _{k=1}^{n}
\Bigg\{\Bigg[v_{k}^{2}b_k^2  + 
\sqrt{2}v_{k}b_kc^{k+3}_{k}(\beta^{1}_{k} - 
\beta^{3}_{k}b_k) + 
\big(c^{k+3}_{k}\big)^2\beta^{3}_{k}b_k
\Big(\dfrac{\beta^{3}_{k}b_k}{2} -
\beta^{1}_{k}\Big)\Bigg]
\Big(\delta_{l1} - \delta_{l2}\Big) \\
&+& \Bigg[ v_k^2b_k - 
\dfrac{b_k\big(c^{k+3}_{k}\beta^{3}_{k}\big)^2}{2} \Bigg]\Big(
\delta_{l3} + \delta_{l(n+4)}\Big) \\  \nonumber
&+&
\displaystyle\sum _{s=1}^{n}\Big\{
\sqrt{2}v_{k}b_kc^{s+3}_{k} - c^{s+3}_{k}c^{k+3}_{k}\beta^{3}_{k}b_k \Big\}\delta_{l(s+3)}\Bigg\}
 + 
\psi^l, \ \ \forall l=1,\ldots,(n+4).
\end{eqnarray}

Since $Y$ and $\eta$ are determined by \eqref{Expressao depois da mudança do Y} and \eqref{Expressao depois da mudança no eta}, our next goal is to verify that the remaining metric relations are satisfied. This process will yield necessary conditions for the constants of integration that were introduced during the calculations.

%Observe que já fizemos uso de quase todas as 
%relações de ortogonalidade da métrica $g$, falta 
%usar que
%$Y,\eta\in C^{n+3}\subset\mathbb{R}^{n+4}
%_2$ são tipo luz e $\langle Y,\eta\rangle=0$.
%A partir disso, podemos reescrever 
%\eqref{Expressao depois da mudança do Y} e 
%\eqref{Expressao depois da mudança no eta},
%obtendo assim uma expressão final para $Y^l$ e $\eta^l$,
%$l=1,\ldots,(n+4)$ nas duas próximas proposições.

\begin{proposition}\label{Lema da conta gigante das constantes}
Let $x:M^n\rightarrow\mathbb{R}^{n+1}$ be an
L-isotropic hypersurface with $n$ distinct nonzero principal curvatures. Consider 
$Y:M^n\rightarrow C^{n+3}\subset
\mathbb{R}^{n+4}_2$ the Laguerre position vector of $x$, 
with the Laguerre metric $g=\langle dY,dY \rangle$, and 
$\eta:M^n\rightarrow C^{n+3}\subset
\mathbb{R}^{n+4}_2$ the Laguerre normal map.
Suppose that $M^n$ is connected and admits a parametrization by lines
of curvature with respect to the metric $g$. Then, the coordinate functions of 
$Y$ are given by,

\begin{eqnarray}\nonumber\label{Yl depois de todo ajeitamento}
Y^{l}&=&
\displaystyle\sum _{k=1}^{n}
\Bigg\{\Bigg[v_{k}^{2}b_k  + 
\sqrt{2}v_{k}c^{k+3}_{k}(\beta^{1}_{k} - 
\beta^{3}_{k}b_k) + 
\big(c^{k+3}_{k}\big)^2\beta^{3}_{k}
\Big(\dfrac{\beta^{3}_{k}b_k}{2} -
\beta^{1}_{k}\Big) +  \gamma^{1}_{k} \Bigg]
\Big(\delta_{l1} - \delta_{l2}\Big)\Bigg\} \\ \nonumber
&+& \Bigg[ \displaystyle\sum _{k=1}^{n}v_k^2 - 
\dfrac{1}{2} \Bigg]\Big(
\delta_{l3} + \delta_{l(n+4)}\Big) \\ 
&+&
\displaystyle\sum _{k,s=1}^{n}\Big\{
\sqrt{2}v_{k}c^{s+3}_{k}\Big\}\delta_{l(s+3)}
 + 
\delta_{l(n+4)}, \ \ \forall l=1,\ldots,(n+4),
\end{eqnarray}
\end{proposition}
\begin{proof}
Using \eqref{F1i=-F2i}, 
\eqref{Y nao sei mais o que colocar} and $\langle Y,Y\rangle=0$ we have,
\begin{eqnarray}\label{nem sei mais o que colocar 2}
%&&-(Y^1)^2 + (Y^1)^2 + (Y^3)^2 + 
%\displaystyle\sum _{s=1}^{n}(Y^{s+3})^2 - (1+ 
%Y^3)^2=0\\
&&\displaystyle\sum _{s=1}^{n}(Y^{s+3})^2 -2Y^3- 1
=0,
\end{eqnarray}
where now $Y$ is a function of $(v_1, \ldots, v_n)$.
Differentiating \eqref{nem sei mais o que colocar 2}
with respect to $v_i$, $i=1,\ldots,n$, we obtain,
\begin{equation}\label{Y resumido primeira derivada}
\displaystyle\sum _{s=1}^{n}Y^{s+3}Y^{s+3}_{,i} -Y^3_{,i} =0
\end{equation}
Then, from \eqref{Expressao depois da mudança do Y}, it follows that,
\begin{equation}\label{correcao nova 25-04...}
    Y^{s+3}_{,i} = \sqrt{2}c^{s+3}_{i}\, \, \, \, \, \, \text{and}\, \, \, \, \, \, Y^{3}_{,i}=2v_i.
\end{equation}
Differentiating again \eqref{Y resumido primeira derivada} with respect to 
$v_j$, $j \neq i = 1,\ldots,n$, we obtain,
\begin{eqnarray}\label{ck.cs=0}
%&&\displaystyle\sum _{s=1}^{n}Y^{s+3}_{,j}Y^{s+3}_{,i} = 0 \\
&& \displaystyle\sum _{s=1}^{n}c^{s+3}_{j}c^{s+3}_{i}= 0, \, \, \, \,\, \forall i\neq j=1,\ldots,n.
\end{eqnarray}
Combining \eqref{ck.cs=0} with the fact that
$ \displaystyle\sum _{s=1}^{n}\big(c^{s+3}_{i}\big)^2 =1$, for all $i=1,\ldots,n,$ we have

%\begin{eqnarray}\label{fato de C ser ortogonal}
%\nonumber
%&&\displaystyle\sum _{l=1}^{n}\big(c^{l+3}_{k}\big)^2 =1,\ \ \forall k=1,\ldots,n, \\
%\\ \nonumber
%&& \displaystyle\sum _{l=1}^{n}
%\Big(c^{l+3}_{k}c^{l+3}_{s}
%\Big) = 0, \ \ \forall s\neq k=1,\ldots,n,
%\end{eqnarray}
\begin{equation}\label{lembrando novamente que ciquadrado e 1}
    \displaystyle\sum _{s=1}^{n}c^{s+3}_{i}c^{s+3}_{j} =\delta_{ij},\, \, \, \,\, \forall i,j=1,\ldots,n.
\end{equation}
In other words, the matrix $\mathcal{C}_{n\times n}$ 
composed of all the constants $c^{s+3}_{i}$ is orthogonal, 
i.e., $\mathcal{C}$ is invertible and 
$\mathcal{C}^{-1} = \mathcal{C}^{T},$ where
\begin{eqnarray}\label{matriz das constantes C}
\mathcal{C}_{n\times n} =  
\left(\begin{array}{cccccc}
c^{4}_{1}& c^{5}_{1}  & c^{6}_{1} &
\ldots & c^{n+3}_{1} \\
c^{4}_{2}& c^{5}_{2}  & c^{6}_{2} &
\ldots & c^{n+3}_{2}\\
c^{4}_{3}& c^{5}_{3}  & c^{6}_{3} &
\ldots & c^{n+3}_{3}\\
\vdots & \vdots & \vdots & \ddots &\vdots \\
c^{4}_{n}& c^{5}_{n}  & c^{6}_{n} &
\ldots & c^{n+3}_{n}\
\end{array}\right)_{n\times n}.
\end{eqnarray}
From now on, whenever the matrix $\mathcal{C}$ is mentioned, we will be 
referring to \eqref{matriz das constantes C}.

Now, substituting \eqref{correcao nova 25-04...} and
\eqref{lembrando novamente que ciquadrado e 1} into 
\eqref{Y resumido primeira derivada}, we obtain,

%\begin{eqnarray} \nonumber
%&& \displaystyle\sum _{s=1}^{n}Y^{s+3}Y^{s+3}_{,i} -Y^3_{,i} =0 \\ \nonumber
%&&\displaystyle
%\sum _{s=1}^{n}\sqrt{2}c^{s+3}_{i}Y^{s+3} - 2v_i=0
%\\ \nonumber
%&&\displaystyle
%\sum _{s,k=1}^{n}
%\Bigg\{2v_kc^{s+3}_{i}c^{s+3}_{k} 
%- \sqrt{2}c^{s+3}_{i}c^{s+3}_{k}c^{k+3}_{k}
%\beta^3_k + \sqrt{2}c^{s+3}_{i}\alpha^{s+3}_k
%\Bigg\} - 2v_i=0 \\ \nonumber
%&&\displaystyle\sum _{k=1}^{n}2v_k\delta_{ik} -
%\displaystyle\sum _{k=1}^{n}
%\sqrt{2}c^{k+3}_{k}
%\beta^3_k\delta_{ik}  + \displaystyle
%\sum _{s,k=1}^{n}\sqrt{2}c^{s+3}_{i}\alpha^{s+3}_k
%- 2v_i = 0.\\ \nonumber
%\end{eqnarray}
%Dessa forma, obtemos
\begin{equation}\label{novacriadacorrecao}
\displaystyle
\sum _{s,k=1}^{n}c^{s+3}_{i}\alpha^{s+3}_k =
c^{i+3}_{i}\beta^3_i,\, \, \, \, \, \forall i=1,\ldots,n.
\end{equation}
The expression \eqref{novacriadacorrecao} can be rewritten in matrix form as,
\begin{equation}\label{eq. matriz para alfa}
\mathcal{C}_{n\times n}\mathcal{\alpha}_{n\times 1}
=\mathcal{A}_{n\times 1},
\end{equation}
where 
\begin{eqnarray*}
\mathcal{\alpha}_{n\times 1} =  
\left(\begin{array}{c}
\displaystyle\sum _{k=1}^{n}\alpha^{4}_{k}\\
\displaystyle\sum _{k=1}^{n}\alpha^{5}_{k}\\
\displaystyle\sum _{k=1}^{n}\alpha^{6}_{k}\\
\vdots \\
\displaystyle\sum _{k=1}^{n}\alpha^{n+3}_{k}
\end{array}\right)_{n\times 1}, \ \ \ \ 
\mathcal{A}_{n\times 1} =  
\left(\begin{array}{c}
c^{4}_{1}\beta_{1}^3\\
c^{5}_{2}\beta_{2}^3\\
c^{6}_{3}\beta_{3}^3\\
\vdots \\
c^{n+3}_{n}\beta_{n}^3\\
\end{array}\right)_{n\times 1}.
\end{eqnarray*}
Since $\mathcal{C}$ is orthogonal, it follows from 
\eqref{eq. matriz para alfa} that
\begin{eqnarray}\label{alfa=Ct.A}
%&&\mathcal{C}_{n\times n}\mathcal{\alpha}_{n\times 1}
%=\mathcal{A}_{n\times 1}\\ \nonumber
%&&\mathcal{C}^{-1}_{n\times n}
%\mathcal{C}_{n\times n}\mathcal{\alpha}_{n\times 1}
%=\mathcal{C}^{-1}_{n\times n}
%\mathcal{A}_{n\times 1}\\
&&\mathcal{\alpha}_{n\times 1}
=\mathcal{C}^{T}_{n\times n}
\mathcal{A}_{n\times 1},
\end{eqnarray}
i.e., 
\begin{equation}\label{expressão final dos alfa}
\displaystyle\sum _{k=1}^{n}
 \Big(\alpha^{s+3}_{k}\Big) = 
 \displaystyle\sum _{k=1}^{n}
 \Big(\beta_{k}^3c^{s+3}_{k}c^{k+3}_{k}\Big), 
\ \ \ \forall s=1,\ldots,n.
\end{equation}
Note that each $s$ indicates a row of the 
matrix given by equation \eqref{alfa=Ct.A}.

Now, from \eqref{Expressao depois da mudança do Y} and
\eqref{expressão final dos alfa}, we have,
\begin{eqnarray}\label{Ys+2 correcao23....}
    Y^{s+3} = \displaystyle\sum _{k=1}^{n}\sqrt{2}v_{k}c^{s+3}_{k}.
\end{eqnarray}
From \eqref{Expressao depois da mudança do Y},
\begin{eqnarray}\label{Y3 correcao23....}
    Y^{3} = \displaystyle\sum _{k=1}^{n}\Bigg[v_k^2 - 
    \dfrac{\big(c^{k+3}_{k}\beta^{3}_{k}\big)^2}{2}
 + \gamma^{3}_{k}\Bigg].
\end{eqnarray}
Note that, using \eqref{lembrando novamente que ciquadrado e 1}, we have,
\begin{eqnarray}\label{Ys+2^2 correcao23....}
   \displaystyle\sum _{s=1}^{n}(Y^{s+3})^2 = 
   \displaystyle\sum _{s=1}^{n}\Bigg[
   \displaystyle\sum _{k=1}^{n}\sqrt{2}v_{k}c^{s+3}_{k}\displaystyle\sum _{j=1}^{n}\sqrt{2}v_{j}c^{s+3}_{j} \Bigg]= 
   2\displaystyle\sum _{s,k,j=1}^{n}v_{k}v_{j}c^{s+3}_{j}c^{s+3}_{k} =
   2\displaystyle\sum _{k=1}^{n}v_{k}^2.
\end{eqnarray}

Substituting \eqref{Ys+2 correcao23....},
\eqref{Y3 correcao23....}, and 
\eqref{Ys+2^2 correcao23....} into
\eqref{nem sei mais o que colocar 2}, we obtain,
\begin{eqnarray}\nonumber
&&  2\displaystyle\sum _{k=1}^{n}v_{k}^2 - 
2\displaystyle\sum _{k=1}^{n}\Bigg[v_k^2 - 
\dfrac{\big(c^{k+3}_{k}\beta^{3}_{k}\big)^2}{2} 
+ \gamma^{3}_{k}  \Bigg] -1 =0.
\end{eqnarray}
%\begin{eqnarray*}
%&& \displaystyle\sum _{l,k,s=1}^{n}
% \Bigg[
%\Big(
%\alpha^{l+3}_{k} -
%\beta_{k}^3c^{l+3}_{k}c^{k+3}_{k} \Big)
%\Big(\alpha^{l+3}_{s} -
%\beta_{s}^3c^{l+3}_{s}c^{s+3}_{s} 
%\Big)\Bigg] +\sum _{k=1}^{n}\Big(
%\big(c^{k+3}_{k}\beta^{3}_{k}\big)^2 - %2\mu^{3}_{k} \Big)-1
%= 0\\
%&& \displaystyle\sum _{l=1}^{n}
% \Bigg[ \Bigg(\displaystyle\sum _{k=1}^{n}
%\Big(
%\alpha^{l+3}_{k} -
%\beta_{k}^3c^{l+3}_{k}c^{k+3}_{k} \Big)\Bigg)
%\Bigg(\displaystyle\sum _{s=1}^{n}
%\Big(\alpha^{l+3}_{s} -
%\beta_{s}^3c^{l+3}_{s}c^{s+3}_{s} 
%\Big)\Bigg) \Bigg] 
%+\sum _{k=1}^{n}\Big(
%\big(c^{k+3}_{k}\beta^{3}_{k}\big)^2 - 2\mu^{3}_{k} \Big)-1
%= 0\\
%&& \sum _{k=1}^{n}\Big(
%\big(c^{k+3}_{k}\beta^{3}_{k}\big)^2 - 2\mu^{3}_{k} \Big)-1
%= 0.
%\end{eqnarray*}
Then,
\begin{equation}\label{condicao do gama3k}
\sum _{k=1}^{n}\big(\gamma^{3}_{k}\big)=
\sum_{k=1}^{n}
\Bigg[
\dfrac{\big(c^{k+3}_{k}\beta^{3}_{k}\big)^2}{2}  
\Bigg]  - \dfrac{1}{2}.
\end{equation}
Therefore, using \eqref{expressão final dos alfa}
and \eqref{condicao do gama3k} in \eqref{Expressao depois da mudança do Y}, 
we obtain \eqref{Yl depois de todo ajeitamento}.
\end{proof}

Next, using the  properties \eqref{relacoes de ortogonalidade para Y}, we will 
derive conditions on the constants
of integration involved in the coordinate functions of $\eta$.

\begin{proposition}
Under the same conditions as in Proposition 
\ref{Lema da conta gigante das constantes}, 
the coordinate functions of 
$\eta \in C^{n+3} \subset \mathbb{R}^{n+4}_2$ 
are given as follows,

\begin{eqnarray}\nonumber\label{eta correcao resumido final}
\eta^{l}&=& 
 \displaystyle\sum _{k=1}^{n}
\Bigg[v_{k}^{2}b_k^2  + 
\sqrt{2}v_{k}b_kc^{k+3}_{k}(\beta^{1}_{k} - 
\beta^{3}_{k}b_k) + 
\dfrac{(c^{k+3}_{k}(\beta^{1}_{k} - 
\beta^{3}_{k}b_k))^2}{2} + \dfrac{1}{2}\epsilon
\Bigg]
\Big(\delta_{l1} - \delta_{l2}\Big) \\ \nonumber
&+&\displaystyle\sum _{k=1}^{n} \Bigg\{ v_k^2b_k - 
\Bigg[ \big(c^{k+3}_{k}\big)^2\beta^{3}_{k}
\Bigg(\dfrac{\beta^{3}_{k}b_k}{2} - \beta^{1}_{k} \Bigg) + \gamma^{1}_{k}
\Bigg]\Bigg\}\Big(
\delta_{l3} + \delta_{l(n+4)}\Big) \\ 
&+&
\displaystyle\sum _{k,s=1}^{n}\Big\{
\sqrt{2}v_{k}b_k + c^{k+3}_{k}(\beta^{1}_{k} - 
\beta^{3}_{k}b_k) \Big\}c^{s+3}_{k}\delta_{l(s+3)}, \ \ \forall l=1,\ldots,(n+4),
\end{eqnarray}
where $\epsilon = 1$, if $l=1$ and $\epsilon = -1$, if $l=2$.    
\end{proposition}
\begin{proof}
Since $\langle \eta, P \rangle = -1$ and \eqref{Expressao depois da mudança no eta} 
holds, it follows that $\eta^1 + \eta^2 = 1$ and $\psi^1 + \psi^2 = 1$. Now, since $\langle \eta, \eta \rangle = 0$, using \eqref{eta3=etan+4}, we have
\begin{eqnarray} \label{eta luz}
%&&-(\eta^1)^2 + (\eta^2)^2 + (\eta^3)^2 + 
%\displaystyle\sum _{s=1}^{n}(\eta^{s+3})^2 
%- (\eta^{n+4})^2=0\\  \nonumber
%&&-(\eta^1)^2 + (1 - \eta^1)^2 + 
%\displaystyle\sum _{s=1}^{n}(\eta^{s+3})^2 =0\\
&&1 -2\eta^1 + 
\displaystyle\sum _{s=1}^{n}(\eta^{s+3})^2 =0,
\end{eqnarray}
where $\eta$ is a function of $(v_1,v_2,\ldots,v_n)$.

Differentiating \eqref{eta luz} with respect to 
$v_i$, $i = 1, \ldots, n$, we obtain
\begin{equation}\label{derivada do eta resumido correcao}
\displaystyle\sum _{s=1}^{n}\eta^{s+3}\eta^{s+3}_{,i} - \eta^1_{,i}=0.
\end{equation}
It follows from \eqref{Expressao depois da mudança no eta} and 
\eqref{derivada do eta resumido correcao} that
\begin{equation}\label{derivada expressao do eta corerecao}
\eta^{s+3}_{,i} = \sqrt{2}b_ic^{s+3}_{i},\, \, \, \, \, \text{e},\, \, \, \, \, \eta^1_{i}= 2v_ib_i^2 + \sqrt{2}b_ic_i^{i+3}(\beta^1_i - \beta^3_ib_i). 
\end{equation}
Substituting \eqref{Expressao depois da mudança no eta} 
(for $l = s + 3$) and \eqref{derivada expressao do eta corerecao} into \eqref{derivada do eta resumido correcao}, and using \eqref{lembrando novamente que ciquadrado e 1}, we obtain,
\begin{eqnarray*}
%&\displaystyle\sum _{s=1}^{n}\Bigg\{\sqrt{2}b_ic^{s+3}_{i}\Bigg[\psi^{s+3}
%+ \displaystyle\sum _{k=1}^{n}\Big(
%\sqrt{2}v_{k}b_kc^{s+3}_{k} - c^{s+3}_{k}c^{k+3}_{k}\beta^{3}_{k}b_k \Big)\Bigg]
%\Bigg\} - 2v_ib_i^2 - \sqrt{2}b_ic_i^{i+3}(\beta^1_i - \beta^3_ib_i)=0\\
&\displaystyle\sum _{s=1}^{n}\sqrt{2}b_ic^{s+3}_{i}\psi^{s+3} +
2\displaystyle\sum _{k=1}^{n}v_{k}b_kb_i\delta_{ik} -
\sqrt{2}\displaystyle\sum _{k=1}^{n}\delta_{ik}c^{k+3}_{k}\beta^{3}_{k}b_kb_i - 2v_ib_i^2 - \sqrt{2}b_ic_i^{i+3}(\beta^1_i - \beta^3_ib_i)=0.
%&&\displaystyle\sum _{s=1}^{n}\sqrt{2}b_ic^{s+3}_{i}\psi^{s+3} +
%2\displaystyle\sum _{s=1}^{n}v_{i}b_i^2(c^{s+3}_{i})^2 -
%\sqrt{2}\displaystyle\sum _{s=1}^{n}(c^{s+3}_{i})^2c^{i+3}_{i}\beta^{3}_{i}b_i^2 - 2v_ib_i^2 - \sqrt{2}b_ic_i^{i+3}(\beta^1_i - \beta^3_ib_i)=0.
\end{eqnarray*}
Then,
%temos que a última igualdade se reduz a,
%\begin{eqnarray*}
%\displaystyle\sum _{s=1}^{n}\sqrt{2}b_ic^{s+3}_{i}\psi^{s+3} - \sqrt{2}b_ic_i^{i+3}\beta^1_i=0    
%\end{eqnarray*}
since $b_i \neq 0$, $\forall i = 1, \ldots, n$, the last equality reduces to,
\begin{equation}
 \displaystyle\sum _{s=1}^{n}c^{s+3}_{i}\psi^{s+3} =
 c_i^{i+3}\beta^1_i,\, \, \, \, \, \forall i=1,\ldots,n. 
\end{equation}

\noindent We rewriting this last expression in matrix form,
\begin{equation}\label{eq. matriz para teta}
\mathcal{C}_{n\times n}\mathcal{\psi}_{n\times 1}
=\mathcal{D}_{n\times 1},
\end{equation}
where 
\begin{eqnarray*}
\mathcal{\psi}_{n\times 1} =  
\left(\begin{array}{c}
\psi^4\\
\psi^5\\
\psi^6\\
\vdots \\
\psi^{n+3}
\end{array}\right)_{n\times 1}, \ \ \ \ 
\mathcal{D}_{n\times 1} =  
\left(\begin{array}{c}
c^{4}_{1}\beta_{1}^1\\
c^{5}_{2}\beta_{2}^1\\
c^{6}_{3}\beta_{3}^1\\
\vdots \\
c^{n+3}_{n}\beta_{n}^1\\
\end{array}\right)_{n\times 1}.
\end{eqnarray*}
Since $\mathcal{C}$ is orthogonal, it follows from
\eqref{eq. matriz para teta} that
\begin{eqnarray}\nonumber\label{teta=Ct.D}
&&\mathcal{\psi}_{n\times 1}
=\mathcal{C}^{t}_{n\times n}
\mathcal{D}_{n\times 1},
\end{eqnarray}
i.e., 
\begin{equation}\label{expressão final dos teta}
\psi^{s+3} = 
 \displaystyle\sum _{k=1}^{n}
 \Big(c^{s+3}_{k}c^{k+3}_{k}\beta_{k}^1\Big), 
\ \ \ \forall s=1,\ldots,n.
\end{equation}
%Observe que cada $s$ indica uma linha da 
%matriz dada pela equação \eqref{teta=Ct.D}.
Considering \eqref{eta luz} and using
\eqref{Expressao depois da mudança no eta},
\eqref{lembrando novamente que ciquadrado e 1}, and 
\eqref{expressão final dos teta}, we obtain,
\begin{eqnarray} \nonumber\label{conta enorme de substituir na eq sem derivar eta}
&&\displaystyle\sum _{s=1}^{n}
\Bigg[\displaystyle\sum _{k=1}^{n}\Big(
\sqrt{2}v_kb_{k}c^{s+3}_{k} - 
c^{s+3}_{k}c^{k+3}_{k}\beta_{k}^3b_{k}  
\Big)+ \psi^{s+3}\Bigg]^2 -
\sum _{k=1}^{n}\Bigg[
2v_{k}^2b_{k}^2 +\\ \nonumber 
&&+ 2\sqrt{2}v_{k}b_{k}c^{k+3}_{k}
(\beta^{1}_{k}-\beta^{3}_{k}b_k)  +
+2(c^{k+3}_{k})^2\beta^{3}_{k}b_{k}\Bigg(\dfrac{\beta^{3}_{k}b_{k}}{2} - \beta^{1}_{k}\Bigg)
\Bigg] -2\psi^{1} +1 =0.
\end{eqnarray}
Thus,
%\begin{eqnarray} \nonumber
%&&2 \displaystyle\sum _{k,t=1}^{n}v_kv_tb_{k}b_{t}\delta_{kt} 
%+2\sqrt{2}\displaystyle\sum _{t,k=1}^{n}
%\Big[v_tb_tc^{k+3}_k (\beta^{1}_{k}-\beta^{3}_{k}b_k)\delta_{kt}\Big] +\\ \nonumber
%&&+\displaystyle\sum _{t,k=1}^{n}
%\Big[c^{k+3}_kc^{t+3}_t
%(\beta^{1}_{k}-\beta^{3}_{k}b_k)
%(\beta^{1}_{t}-\beta^{3}_{t}b_t)\delta_{kt}\Big]
%-\sum _{k=1}^{n}\Bigg[
%2v_{k}^2b_{k}^2 + 2\sqrt{2}v_{k}b_{k}c^{k+3}_{k}
%(\beta^{1}_{k}-\beta^{3}_{k}b_k) + \\ \nonumber
%&&+2(c^{k+3}_{k})^2\beta^{3}_{k}b_{k}\Bigg(\dfrac{\beta^{3}_{k}b_{k}}{2} - \beta^{1}_{k}\Bigg)
%\Bigg] -2\psi^{1} +1 =0,
%\end{eqnarray}
%isto é,
\begin{eqnarray} \nonumber
&&\displaystyle\sum _{k=1}^{n}
\Big[(c^{k+3}_k)^2
(\beta^{1}_{k}-\beta^{3}_{k}b_k)^2\Big]
- \sum _{k=1}^{n}\Bigg[
2(c^{k+3}_{k})^2\beta^{3}_{k}b_{k}\Bigg(\dfrac{\beta^{3}_{k}b_{k}}{2} - \beta^{1}_{k}\Bigg)
\Bigg] -2\psi^{1} +1 =0. \nonumber
%&&\displaystyle\sum _{k=1}^{n}
%\Big(c^{k+3}_{k}(\beta_{k}^1 -
%\beta_{k}^3b_{k})\Big)^2
%+  \sum _{k=1}^{n}\Bigg[
%2(c^{k+3}_{k})^2\beta^{3}_{k}\beta^{1}_{k}b_{k} - 
%\big(c^{k+3}_{k}\beta^{3}_{k}b_{k}\big)^2\Bigg] %-2\psi^{1} +1 =0.
\end{eqnarray}
\noindent Therefore, we obtain
\begin{equation}\label{expressao do psi 1 correçao}
\psi^{1}=
\sum_{k=1}^{n}
\Bigg[
\dfrac{\big(c^{k+3}_{k}\beta^{1}_{k}\big)^2}{2}  
\Bigg]  + \dfrac{1}{2}.
\end{equation}
Using the fact that $\psi^1 + \psi^2 = 1$, 
\eqref{expressão final dos teta}, and 
\eqref{expressao do psi 1 correçao} in 
\eqref{Expressao depois da mudança no eta}, we can
rewrite $\eta^{l}$ as follows,
\begin{eqnarray}\nonumber\label{etal depois da correcao resumido}
\eta^{l}&=& 
 \displaystyle\sum _{k=1}^{n}
\Bigg[v_{k}^{2}b_k^2  + 
\sqrt{2}v_{k}b_kc^{k+3}_{k}(\beta^{1}_{k} - 
\beta^{3}_{k}b_k) + 
\dfrac{(c^{k+3}_{k}(\beta^{1}_{k} - 
\beta^{3}_{k}b_k))^2}{2} + \dfrac{1}{2}\epsilon
\Bigg]
\Big(\delta_{l1} - \delta_{l2}\Big) \\
&+&\Bigg\{\psi^3 + \displaystyle\sum _{k=1}^{n} \Bigg[ v_k^2b_k - 
\dfrac{b_k\big(c^{k+3}_{k}\beta^{3}_{k}\big)^2}{2} \Bigg]\Bigg\}\Big(
\delta_{l3} + \delta_{l(n+4)}\Big) \\  \nonumber
&+&
\displaystyle\sum _{k,s=1}^{n}\Big\{
\sqrt{2}v_{k}b_k + c^{k+3}_{k}(\beta^{1}_{k} - 
\beta^{3}_{k}b_k) \Big\}c^{s+3}_{k}\delta_{l(s+3)}, \ \ \forall l=1,\ldots,(n+4),
\end{eqnarray}
where $\epsilon = 1$, if $l=1$ and $\epsilon = -1$, if $l=2$.

Finally, considering 
$\langle Y, \eta \rangle = 0$ and using the relations $\eqref{F1i=-F2i}$, 
$\eqref{F8 = 1+ F3}$, $\eqref{eta3=etan+4}$, and the fact that 
$\psi^1 + \psi^2 = 1$, we have,
%\begin{eqnarray} \nonumber
%&&-Y^{1}\eta^1 + (-Y^{1})(1-\eta^1) + 
%Y^{3}\eta^3 + 
%\displaystyle\sum _{s=1}^{n}Y^{s+3}\eta^{s+3} 
%- (1+Y^{3})\eta^3 =0,
%\end{eqnarray}
%que se reduz a 
\begin{eqnarray} \label{eta e Y =0}
&&-Y^{1}+
\displaystyle\sum _{s=1}^{n}Y^{s+3}\eta^{s+3} 
- \eta^3 =0.
\end{eqnarray}
Substituting
\eqref{Yl depois de todo ajeitamento} and 
\eqref{etal depois da correcao resumido} into
\eqref{eta e Y =0} and using
\eqref{lembrando novamente que ciquadrado e 1}, we obtain,
%\begin{eqnarray*}\label{conta resumida correcao}\nonumber
%&&\displaystyle\sum _{s=1}^{n}
%\Bigg[\displaystyle\sum _{k=1}^{n}\Big(
%\sqrt{2}v_kc^{s+3}_{k} 
%\Big)\Bigg]
%\Bigg[\displaystyle\sum _{t=1}^{n}\Big(
%\sqrt{2}v_tc^{s+3}_{t}b_{t} + 
%c^{s+3}_{t}c^{t+3}_{t}(\beta_{t}^1-\beta_{t}^3b_{t})
%\Big)\Bigg] + \\ \nonumber
%&&\, \, \, \, \, \, \, \, \, \, - \displaystyle\sum _{k=1}^{n}
%\Bigg[ v_{k}^2b_{k} + 
%\sqrt{2}v_{k}c^{k+3}_{k}(\beta^{1}_{k} 
%- \beta^{3}_{k}b_{k})+
%\big(c^{k+3}_{k}\big)^2\beta^{3}_{k}
%\Big(\dfrac{\beta^{3}_{k}b_{k}}{2} - 
%\beta^{1}_{k}  \Big)
%+  \gamma^{1}_{k}\Bigg]+\\ \nonumber
%&&\, \, \, \, \, \, \, \, \, \,- \displaystyle\sum _{k=1}^{n}
%\Bigg[ v_{k}^2b_{k} - \dfrac{b_{k}(c^{k+3}_{k}\beta^{3}_{k})^2}{2}\Bigg]  - \psi^3
%=0.
%\end{eqnarray*}
%Portanto,
\begin{eqnarray*}
&&\displaystyle\sum _{k,t=1}^{n}\Bigg[
2v_kv_tb_t\delta_{kt} + 
\sqrt{2}v_kc^{t+3}_{t}(\beta_{t}^1-\beta_{t}^3b_{t})\delta_{kt}\Bigg] +\\
&&\, \, \, \, \, \, \, \, \, \, + \displaystyle\sum _{k=1}^{n}
\Bigg[\big(c^{k+3}_{k}\big)^2\beta^{3}_{k}
\beta^{1}_{k} - 2v_{k}^2b_{k} - 
\sqrt{2}v_{k}c^{k+3}_{k}(\beta^{1}_{k} 
- \beta^{3}_{k}b_{k}) -  \gamma^{1}_{k}\Bigg]- \psi^3
=0.
\end{eqnarray*}
%Usando \eqref{lembrando novamente que ciquadrado e 1} segue os cálculos na expressão
%\eqref{conta resumida correcao}, 
%%\begin{eqnarray*}
%&&\displaystyle\sum _{k,t=1}^{n}\Bigg[
%2v_kv_tb_t + 
%\sqrt{2}v_kc^{t+3}_{t}(\beta_{t}^1-\beta_{t}^3b_{t})\Bigg]\delta_{kt} \\
%&& + \displaystyle\sum _{k=1}^{n}
%\Bigg[\big(c^{k+3}_{k}\big)^2\beta^{3}_{k}
%\beta^{1}_{k} - 2v_{k}^2b_{k} - 
%\sqrt{2}v_{k}c^{k+3}_{k}(\beta^{1}_{k} 
%- \beta^{3}_{k}b_{k}) -  \gamma^{1}_{k}\Bigg]- \psi^3
%=0 \\
%&&\displaystyle\sum _{k=1}^{n}
%\Bigg[\big(c^{k+3}_{k}\big)^2\beta^{3}_{k}
%\beta^{1}_{k} -  \gamma^{1}_{k}\Bigg]- \psi^3
%=0.    
%\end{eqnarray*}
Thus, 
\begin{equation}\label{relacao psi 3 depois da correcao}
\psi^3 =  \displaystyle\sum _{k=1}^{n}
\Bigg[\big(c^{k+3}_{k}\big)^2\beta^{3}_{k}
\beta^{1}_{k} -  \gamma^{1}_{k}\Bigg].   
\end{equation}
Finally, substituting \eqref{relacao psi 3 depois da correcao} into \eqref{etal depois da correcao resumido}, we obtain the expression 
\eqref{eta correcao resumido final}.
\end{proof}
In what follows, having obtained the vectors 
$Y, \eta \in \mathbb{R}^{n+4}_2$, we can present an 
explicit expression for $x$ and its unit normal vector $\xi$. Note that this is 
equivalent to 
obtaining, up to a Laguerre transformation, 
the L-isotropic hypersurfaces in 
$\mathbb{R}^{n+1}$ parametrized by lines of curvature. 
More precisely, we have the following result.
\begin{theorem}
Let $x: M^n \rightarrow \mathbb{R}^{n+1}$ be a
hypersurface with $n$ distinct nonzero principal curvatures and 
$Y: M^n \rightarrow C^{n+3} \subset
\mathbb{R}^{n+4}_2$ the Laguerre position vector of $x$ 
with Laguerre metric $g = \langle dY, dY \rangle$, and 
$\lambda$ the eigenvalues of the tensor $\mathbb{L}$.
Suppose that $M^n$ is connected and admits a parametrization by lines
of curvature 
with respect to the metric $g$. If $x$ is an 
L-isotropic hypersurface, then $\lambda = 0$ and $x$ is 
an $L$-isoparametric hypersurface. Moreover, up to a Laguerre transformation, this 
hypersurface is equivalent to the one given by \eqref{definicao da HILF}.
\end{theorem}
\begin{proof}
From Lemma \ref{lema da aplicacao do FroFro}, we have
$\lambda = 0$ and $x$ is an 
$L$-isoparametric hypersurface. It remains to 
verify that $x$ is Laguerre equivalent 
to the hypersurface described in 
\eqref{definicao da HILF}. With this goal in mind, we will
explicitly determine $x$.

By the definition of $Y$ and from 
\eqref{Yl depois de todo ajeitamento}, we have that
\begin{equation}\label{rho da imersão x}
\rho= Y^{n+4} = \dfrac{1}{2} + \displaystyle\sum _{k=1}^{n}v_{k}^{2}.
\end{equation}
Note that $\xi^{q} = \dfrac{Y^{q+2}}{\rho},\ \ 
q = 1, \ldots, (n+1)$, where $\xi^{q}$ are the coordinates
of the unit normal vector field $\xi \in \mathbb{R}^{n+1}$.
Then, from 
\eqref{Yl depois de todo ajeitamento} and \eqref{rho da imersão x}, we obtain that,

\begin{equation}\label{xi com a primeira ro -1} 
\xi^{q}=\dfrac{\rho - 1}{\rho}\delta_{q1} 
+\dfrac{1}{\rho}\Bigg\{ \displaystyle\sum _{k,s=1}^{n}
\sqrt{2}v_{k}c^{s+3}_{k}\delta_{q(s+1)}\Bigg\},
\,  \,  \, \,   q=1,\ldots,(n+1).
\end{equation}

Now, comparing the definition of $\eta$ with the one given in 
\eqref{eta correcao resumido final}, we have
\begin{equation}\label{expressao do r}
r = \displaystyle\sum _{k=1}^{n} \Bigg\{ v_k^2b_k - 
\Bigg[ \big(c^{k+3}_{k}\big)^2\beta^{3}_{k}
\Bigg(\dfrac{\beta^{3}_{k}b_k}{2} - \beta^{1}_{k} \Bigg) + \gamma^{1}_{k}
\Bigg]\Bigg\}.
\end{equation}
%Substituindo na expressão de $r$ a última 
%condição dada em \eqref{todas as condições nas cte}
%obtemos,
%\begin{eqnarray}\nonumber
%&& r = 
%\displaystyle\sum _{k=1}^{n}\Bigg\{v_{k}^{2}b_{k}
%-\dfrac{\big(c^{k+3}_{k}\beta^{3}_{k}\big)^2
%b_{k}}{2}+\theta^{3}_{k}\Bigg\} \\
%&&r = 
%\displaystyle\sum _{k=1}^{n}\Bigg\{v_{k}^{2}b_{k}
%-\dfrac{\big(c^{k+3}_{k}\beta^{3}_{k}\big)^2
%b_{k}}{2}+  
%\big(c^{k+3}_{k}\big)^2\beta^{3}_{k}\beta^{1}_{k}
%-\mu^{1}_{k}\Bigg\}.
%\end{eqnarray}
Moreover, we obtain that the coordinate functions 
of the immersion $x$ are given 
as follows,
\begin{equation}\label{eq do x,eta,xi e r correcao}
    x^{q} = \eta^{q+2} - r\xi^{q},\ \ \ 
q=1,\ldots,(n+1).
\end{equation}
Substituting \eqref{eta correcao resumido final} and 
\eqref{expressao do r} in to \ref{eq do x,eta,xi e r correcao} for all $q=1,\ldots,(n+1)$, we get
%\begin{eqnarray}\nonumber\label{eta correcao q+2}
%&&\eta^{q+2} =  \displaystyle\sum _{k=1}^{n} \Bigg\{ v_k^2b_k - 
%\Bigg[ \big(c^{k+3}_{k}\big)^2\beta^{3}_{k}
%\Bigg(\dfrac{\beta^{3}_{k}b_k}{2} - \beta^{1}_{k} \Bigg) + \gamma^{1}_{k}
%\Bigg]\Bigg\}\delta_{q1} \\  
%&&+\displaystyle\sum _{k,s=1}^{n}\Big\{
%\sqrt{2}v_{k}b_k + c^{k+3}_{k}(\beta^{1}_{k} - 
%\beta^{3}_{k}b_k) \Big\}c^{s+3}_{k}\delta_{q(s+1)}.
%\end{eqnarray}
%Por \eqref{expressao do r}, podemos reescrever 
%a expressão  \eqref{eta correcao q+2} como sendo, para todo 
%$q=1,\ldots,(n+1)$,
\begin{equation}\label{eta(t+2)}
\eta^{q+2} = r\delta_{q1} +\displaystyle\sum _{k,s=1}^{n}\Big\{
\sqrt{2}v_{k}b_k + c^{k+3}_{k}(\beta^{1}_{k} - 
\beta^{3}_{k}b_k) \Big\}c^{s+3}_{k}\delta_{q(s+1)}.
\end{equation}
Then, it follows from 
\eqref{xi com a primeira ro -1}, \eqref{eq do x,eta,xi e r correcao}, 
and \eqref{eta(t+2)} that
%\begin{eqnarray*}
%&&x^{q}= r\delta_{q1} +\displaystyle\sum _{k,s=1}^{n}\Big\{
%\sqrt{2}v_{k}b_k + c^{k+3}_{k}(\beta^{1}_{k} - 
%\beta^{3}_{k}b_k) \Big\}c^{s+3}_{k}\delta_{q(s+1)}+\\ 
%&&-\dfrac{r(\rho - 1)}{\rho}\delta_{q1} -
%\dfrac{r}{\rho}\Bigg\{ \displaystyle\sum _{k,s=1}^{n}
%\sqrt{2}v_{k}c^{s+3}_{k}\delta_{q(s+1)}\Bigg\}.
%&&= \dfrac{r\rho - r(\rho - 1)}{\rho}\delta_{q1}
%+\displaystyle\sum _{k=1}^{n}\Big\{
%\sqrt{2}v_{k}b_k + c^{k+3}_{k}(\beta^{1}_{k} - 
%\beta^{3}_{k}b_k) \Big\}c^{q+3}_{k}\\
%&& -
%\dfrac{r}{\rho}\Bigg\{ \displaystyle\sum _{k=1}^{n}
%\sqrt{2}v_{k}c^{q+2}_{k}\Bigg\}. 
%\end{eqnarray*}
the coordinate functions of the immersion $x$ are given by
\begin{eqnarray}\nonumber \label{funcao coordenada de x final}
\label{funcao coordenada de x}
&&x^{q}= \dfrac{r}{\rho}\delta_{q1}
+
\displaystyle\sum _{k,s=1}^{n}\Bigg\{
\sqrt{2}v_{k}\Big( b_k - \dfrac{r}{\rho}\Big) + c^{k+3}_{k}(\beta^{1}_{k} - 
\beta^{3}_{k}b_k) 
\Bigg\}c^{s+3}_{k}\delta_{q(s+1)}, \ \
\ q=1,\ldots,(n+1). \\
\end{eqnarray}

Finally, consider the following change of coordinates
\begin{eqnarray}\label{mudancanova}
\overline{v_k} = \dfrac{\sqrt{2}v_{k}}{\overline{b_k}},\, \, \, \, \text{onde}\, \, \, \, \overline{b_k}=\dfrac{1}{b_k}.
\end{eqnarray}
Applying \ref{mudancanova} to the expressions of $\rho$ and $r$ given respectively by 
\eqref{rho da imersão x} and \eqref{expressao do r}, we obtain
%\begin{eqnarray}
%\nonumber
%&&\rho= \dfrac{1}{2}\Bigg\{1 + \displaystyle\sum _{k=1}^{n}
%(\overline{v_k}\overline{b_k})^2\Bigg\}.
%\end{eqnarray}
%Agora, fazendo essa mesma mudança na expressão de $r$ dada em \eqref{expressao do r} obtemos,
%\begin{eqnarray}\nonumber\
%&&r = 
% \displaystyle\sum _{k=1}^{n} \Bigg\{ \dfrac{\overline{v_k}^2\overline{b_k}}{2} - 
%\Bigg[ \big(c^{k+3}_{k}\big)^2\beta^{3}_{k}
%\Bigg(\dfrac{\beta^{3}_{k}}{2\overline{b_k}} - \beta^{1}_{k} \Bigg) + \gamma^{1}_{k}
%\Bigg]\Bigg\}.
%\end{eqnarray}

%Portanto, as expressões finais para $\rho$ e $r$ são dadas por,
\begin{eqnarray*}\label{rho final depois da ultima mudança}
\rho= \dfrac{1}{2}\Bigg\{1 + \displaystyle\sum _{k=1}^{n}
(\overline{v_k}\overline{b_k})^2\Bigg\},\, \, \, \, \, r = 
\displaystyle\sum _{k=1}^{n}\Bigg\{\dfrac{\overline{v_k}^2\overline{b_k}}{2}\Bigg\}
+ \phi,
\end{eqnarray*}
where 
\begin{equation}\label{phi final correcao...}
   \phi=
-2\displaystyle\sum _{k=1}^{n} \Bigg\{ 
\Bigg[ \big(c^{k+3}_{k}\big)^2\beta^{3}_{k}
\Bigg(\dfrac{\beta^{3}_{k}}{2\overline{b_k}} - \beta^{1}_{k} \Bigg) + \gamma^{1}_{k}
\Bigg]\Bigg\}. 
\end{equation}
Therefore, using \eqref{mudancanova}, it follows from 
\eqref{xi com a primeira ro -1}
and \eqref{funcao coordenada de x final}
that the immersion $x$ and its normal $\xi$ are given, for all 
$q = 1, \ldots, (n+1)$, by
\begin{eqnarray}\label{finalmente x depois de mudanca correcao}\nonumber
&&x^{q}= \dfrac{r}{\rho}\delta_{q1}
+
\displaystyle\sum _{k,s=1}^{n}\Bigg\{
\overline{v_{k}}\Big(1 - \dfrac{r\overline{b_{k}}}{\rho}\Big) + c^{k+3}_{k}\Big(\beta^{1}_{k} -
 \dfrac{\beta^{3}_{k}}{\overline{b_{k}}}\Big) 
\Bigg\}c^{s+3}_{k}\delta_{q(s+1)},\ \ \ \ q=1,\ldots,(n+1), \\
\end{eqnarray}
and
\begin{eqnarray}
\xi^{q}=\dfrac{1}{2\rho}\Bigg\{\Bigg[
\displaystyle\sum _{k=1}^{n}(\overline{v_{k}}\overline{b_{k}})^2 - 
1 \Bigg]\delta_{q1} 
+\displaystyle\sum _{k,s=1}^{n}
\overline{v_{k}}\overline{b_{k}}c^{s+3}_{k}\delta_{q(s+1)}\Bigg\},
\,  \,  \, \,   q=1,\ldots,(n+1).    
\end{eqnarray}
%Portanto, 
%\begin{eqnarray}\nonumber \label{FINALMENTE X}
%&&x = \Bigg(0, 
%\displaystyle\sum _{k=1}^{n}\overline{v_{k}}c^{4}_{k},
%\ldots, 
%\displaystyle\sum_{k=1}^{n}
%\overline{v_{k}}c^{n+3}_{k}\Bigg)  - \Bigg[\dfrac{\sum_{k=1}^{n}
%\overline{b_{k}}
%\overline{v_{k}}^2 + \phi} 
%{1 +\sum_{k=1}^{n}(
%\overline{v_{k}}\overline{b_{k}})^2}\Bigg]
%\Bigg(-1, 
%\displaystyle\sum_{k=1}^{n}
%\overline{v_{k}}\overline{b_{k}}c^{4}_{k}\, ,\,  \ldots, 
%\displaystyle\sum_{k=1}^{n}
%\overline{v_{k}}\overline{b_{k}}c^{n+3}_{k}\Bigg) \\  \nonumber
%&&+ \Bigg(0, 
%\displaystyle\sum_{k=1}^{n}
%c^{4}_{k}c^{k+3}_{k}\Big(\beta_{k}^1 -
%\dfrac{\beta^{3}_{k}}{\overline{b_{k}}}\Big),
%\ldots,
%\displaystyle\sum_{k=1}^{n}
%c^{n+3}_{k}c^{k+3}_{k}\Big(\beta_{k}^1 -
%\dfrac{\beta^{3}_{k}}{\overline{b_{k}}}\Big)
%\Bigg),\\
%\end{eqnarray}
%onde $\phi$ é dada por \eqref{phi final correcao...}.

%Além disso, o campo normal unitário é dado por
%$\xi$
%\begin{eqnarray}\nonumber
%\label{FINALMENTE normal xi}
%&&\xi = \Bigg[\dfrac{1} 
%{1 +\sum_{k=1}^{n}(
%\overline{v_{k}}\overline{b_{k}})^2}\Bigg]
%\Bigg(\sum_{k=1}^{n}(
%\overline{v_{k}}\overline{b_{k}})^2 - 1\, ,\,  
%2\displaystyle\sum_{k=1}^{n}
%\overline{v_{k}}\overline{b_{k}}c^{4}_{k}\, ,\,  
%\ldots\, ,\,  
%2\displaystyle\sum_{k=1}^{n}
%\overline{v_{k}}\overline{b_{k}}c^{n+3}_{k}\Bigg).
%\\ 
%\end{eqnarray}
Observe that the constants of integration $\beta^{1}_{k}, \beta^{3}_{k}$, and $\gamma^{1}_{k}$, for all $k = 1, \ldots, n$, can be chosen arbitrarily. The most significant condition is on the constants $c$'s which must be selected so that the matrix 
$\mathcal{C}$ given by \eqref{matriz das constantes C} is orthogonal. This choice determines the tangent vectors of $Y$, according to \eqref{correcao nova 25-04...}. One can then determine $\eta$, and subsequently $x$ and $\xi$. Consequently, choosing these constants yields hypersurfaces that are Laguerre equivalent. It therefore follows that the family of hypersurfaces obtained in \eqref{finalmente x depois de mudanca correcao} and the family defined in \eqref{definicao da HILF} are Laguerre equivalent.

\end{proof}
In what follows we will describe the choice of
constants that reduces the 
hypersurfaces 
\eqref{finalmente x depois de mudanca correcao} to \eqref{definicao da HILF}.
\begin{example}
\rm
First, consider the matrix $\mathcal{C} = \mathcal{I}$ as the identity matrix,
%\begin{eqnarray}
%\mathcal{C}_{n\times n} =  
%\left(\begin{array}{cccccccc}
%c^{4}_{1}& c^{5}_{1}  & c^{6}_{1} &
%\ldots & c^{n+3}_{1} \\
%c^{4}_{2}& c^{5}_{2}  & c^{6}_{2} &
%\ldots & c^{n+3}_{2}\\
%c^{4}_{3}& c^{5}_{3}  & c^{6}_{3} &
%\ldots & c^{n+3}_{3}\\
%\vdots & \vdots & \vdots & \ddots &\vdots \\
%c^{4}_{n}& c^{5}_{n}  & c^{6}_{n} &
%\ldots & c^{n+3}_{n}\
%\end{array}\right)_{n\times n} = \, \,  \, \, \, \, 
%\left(\begin{array}{cccccccc}
%1& 0  & 0 &
%\ldots & 0 \\
%0& 1  & 0 &
%\ldots & 0\\
%0& 0  & 1 &
%\ldots & 0\\
%\vdots & \vdots & \vdots & \ddots &\vdots \\
%0& 0  & 0 &
%\ldots & 1\
%\end{array}\right)_{n\times n}
%\end{eqnarray}
With this choice, the immersion $x$ obtained in 
\eqref{finalmente x depois de mudanca correcao}
reduces to,
\begin{eqnarray}\nonumber\label{x correcao exemplo}
&&x = \Bigg(0, 
\overline{v_{1}}, \overline{v_{2}},
\ldots, 
\overline{v_{n}}\Bigg) 
- \Bigg[\dfrac{\sum_{k=1}^{n}
\overline{b_{k}}
\overline{v_{k}}^2 + \phi} 
{1 +\sum_{k=1}^{n}(
\overline{v_{k}}\overline{b_{k}})^2}\Bigg]
\Bigg(-1, 
\overline{v_{1}}\overline{b_{1}}, 
\overline{v_{2}}\overline{b_{2}}, \ldots, 
\overline{v_{n}}\overline{b_{n}}\Bigg) \\
&&+ \Bigg(0, 
\Big(\beta_{1}^1 -
\dfrac{\beta^{3}_{1}}{\overline{b_{1}}}\Big),
\Big(\beta_{2}^1 -
\dfrac{\beta^{3}_{2}}{\overline{b_{2}}}\Big),
\ldots,
\Big(\beta_{n}^1 -
\dfrac{\beta^{3}_{n}}{\overline{b_{n}}}\Big)
\Bigg),
\end{eqnarray}
where
\begin{equation*}
    \phi=
-2\displaystyle\sum _{k=1}^{n} \Bigg\{ 
\Bigg[ \beta^{3}_{k}
\Bigg(\dfrac{\beta^{3}_{k}}{2\overline{b_k}} - \beta^{1}_{k} \Bigg) + \gamma^{1}_{k}
\Bigg]\Bigg\}.
\end{equation*}
The unit normal vector field is given by
\begin{eqnarray}
&&\xi = \Bigg[\dfrac{1} 
{1 +\sum_{k=1}^{n}(
\overline{v_{k}}\overline{b_{k}})^2}\Bigg]
\Bigg(\sum_{k=1}^{n}(
\overline{v_{k}}\overline{b_{k}})^2 - 1, 
2\overline{v_{1}}\overline{b_{1}}, 
2\overline{v_{2}}\overline{b_{2}}, \ldots, 
2\overline{v_{n}}\overline{b_{n}}\Bigg). 
\end{eqnarray}
%Note que, basta escolher as constantes $c's$ para
%obter o mesmo normal dado em 
%\eqref{normal da xi linha exemplo}.

Now, considering
\begin{equation*}
\beta_{k}^1 -
\dfrac{\beta^{3}_{k}}{\overline{b_{k}}} = 0,\, \, \, \, \, \forall\, k=1,\ldots,n
\end{equation*}
in \eqref{x correcao exemplo}, we have,
\begin{eqnarray}\label{x correcao exemplo considerando beta =0...}
&&x = \Bigg(0, 
\overline{v_{1}}, \overline{v_{2}},
\ldots, 
\overline{v_{n}}\Bigg) 
- \Bigg[\dfrac{\sum_{k=1}^{n}
\overline{b_{k}}
\overline{v_{k}}^2 + \phi} 
{1 +\sum_{k=1}^{n}(
\overline{v_{k}}\overline{b_{k}})^2}\Bigg]
\Bigg(-1, 
\overline{v_{1}}\overline{b_{1}}, 
\overline{v_{2}}\overline{b_{2}}, \ldots, 
\overline{v_{n}}\overline{b_{n}}\Bigg),
\end{eqnarray}
where
$\phi=
\displaystyle\sum _{k=1}^{n} \Bigg\{ 
\dfrac{\big(\beta^{3}_{k}\big)^2}{\overline{b_k}} - 2\gamma^{1}_{k}\Bigg\}.$
Setting $\phi = 0$, we obtain \eqref{definicao da HILF}.
\end{example}

\begin{remark}
\rm In \cite{artigo do exemplo}, Corro, A.V.; Ferreira, W.; and Tenenblat, K obtained the hypersurface described in \eqref{x correcao exemplo considerando beta =0...} as an example of a family of Dupin hypersurfaces. 
At first glance, the hypersurface \eqref{x correcao exemplo considerando beta =0...} 
may appear to be a generalization of 
\eqref{definicao da HILF}, where the two differ 
in the choice of $\phi$. However, from our calculations we see that choosing the constant $\phi$ is 
nothing more than choosing the position vector $Y$, which implies equivalence in the Laguerre sense. Thus, from the point of view of Laguerre geometry,
\eqref{x correcao exemplo considerando beta =0...}
and \eqref{definicao da HILF} are equivalent.
\end{remark}

Note that throughout this paper we have obtained 
results concerning L-isotropic hypersurfaces in 
$\mathbb{R}^{n+1}$ with $\lambda = 0$. 
For future work,
we would like to further explore the world of L-isotropic hypersurfaces when 
$\lambda > 0$.

\end{document}